\input ./style/arxiv-general.cfg
\documentclass[aop,MSNbibl,seceqn,nameyear,dvips]{arximspdf}
\makeatletter
   \@ifpackageloaded{graphicx}{}{\usepackage{graphicx}}
\makeatother
\usepackage{mathrsfs}
\usepackage{tikz,pgf}
\usepackage{subfigure}

\doi{10.1214/15-AOP1010}
\volume{44}
\issue{3}
\pubyear{2016}
\firstpage{1776}
\lastpage{1816}
\docsubty{FLA}

\makeatletter

\def\vfrac#1#2{(#1)/#2}

\def\sklvfrac#1#2{((#1)/#2)}

\newcommand{\rrvert}{\vert}
\newcommand{\rrVert}{\Vert}
\newcommand{\llvert}{\vert}
\newcommand{\llVert}{\Vert}
\def\dlp{d_{\mathrm{LP}}}
\newcommand{\urn}{N}
\newcommand{\eqref}[1]{(\ref{#1})}
\newtheorem{theorem}{Theorem}[section]
\newtheorem{proposition}[theorem]{Proposition}
\newtheorem{lemma}[theorem]{Lemma}
\newproclaim{definition}[theorem]{Definition}
\newproclaim{remark}[theorem]{Remark}
\newcommand{\eps}{\varepsilon}
\newcommand{\eq}{\eqref}
\newcommand{\dtv}{d_{\mathrm{TV}}}
\newcommand{\dk}{d_{\mathrm{K}}}
\newcommand{\bigo}{\mathrm{O}}
\newcommand{\lito}{\mathrm{o}}
\newcommand{\Bi}{\operatorname{Bi}}
\newcommand{\GG}{\operatorname{GG}}
\newcommand{\Ga}{{\mathrm{G}}}
\newcommand{\Ge}{\operatorname{Ge}}
\newcommand{\Be}{\operatorname{Be}}
\newcommand{\U}{\mathrm{U}}
\newcommand{\IE}{\mathbb{E}}
\newcommand{\IP}{\mathbb{P}}
\newcommand{\law}{\mathscr{L}}
\newcommand{\IZ}{\mathbb{Z}}
\newcommand{\IR}{\mathbb{R}}
\newcommand{\I}{{\mathrm{I}}}
\def\mid{|}
\def\bmid{|}
\def\G{\Gamma}
\def\B{\mathrm{B}}
\def\ed{\stackrel{\mathscr{D}}{=}}
\def\Gez{\Ge_0}
\def\Geo{\Ge_1}
\makeatother

\begin{document}
\begin{frontmatter}

\title{Generalized gamma approximation with rates for~urns, walks and trees}
\runtitle{Generalized gamma approximation with rates}

\begin{aug}
\author[A]{\fnms{Erol A.}~\snm{Pek\"oz}\thanksref{T1}\ead[label=e1]{pekoz@bu.edu}},
\author[B]{\fnms{Adrian}~\snm{R\"ollin}\thanksref{T1}\ead[label=e2]{adrian.roellin@nus.edu.sg}}
\and
\author[C]{\fnms{Nathan}~\snm{Ross}\corref{}\thanksref{T3}\ead[label=e3]{nathan.ross@unimelb.edu.au}}
\runauthor{E. A. Pek\"oz, A. R\"ollin and N. Ross}
\affiliation{Boston University, National University of Singapore and
University of Melbourne}
\address[A]{E. A. Pek\"oz\\
Questrom School of Business\\
Boston University \\
595 Commonwealth Avenue\\
Boston, Massachusetts 02215\\
USA\\
\printead{e1}}
\address[B]{A. R\"ollin\\
Department of Statistics\\
\quad and Applied Probability\\
National University of Singapore\\
6 Science Drive 2\\
Singapore 117546\\
\printead{e2}}
\address[C]{N. Ross\\
School of Mathematics and Statistics\\
University of Melbourne\\
Parkville VIC, 3010\\
Australia\\
\printead{e3}}
\end{aug}
\thankstext{T1}{Supported in part by NUS Grant R-155-000-124-112.}
\thankstext{T3}{Supported in part by ARC Grant DP150101459, NSF Grants DMS-07-04159, DMS-08-06118, DMS-11-06999 and ONR Grant N00014-11-1-0140.}

%
\received{\smonth{9} \syear{2013}}
%
\revised{\smonth{2} \syear{2015}}

%
\begin{abstract}
We study a new class of time inhomogeneous P\'olya-type urn schemes and
give optimal rates of
convergence for the distribution of the properly scaled number of balls
of a given color to nearly the full class of generalized gamma
distributions with integer parameters, a class which includes the
Rayleigh, half-normal and gamma distributions. Our main tool is Stein's
method combined with characterizing the generalized gamma limiting
distributions as fixed points of distributional transformations related
to the equilibrium distributional transformation from renewal theory.
We identify special cases of these urn models in recursive
constructions of random walk paths and trees, yielding rates of
convergence for local time and height statistics of simple random walk
paths, as well as for the size of random subtrees of uniformly random
binary and plane trees.
\end{abstract}

%
\begin{keyword}[class=AMS]
\kwd[Primary ]{60F05}
\kwd{60C05}
\kwd[; secondary ]{60E10}
\kwd{60K99}
\end{keyword}
\begin{keyword}
\kwd{Generalized gamma distribution}
\kwd{P\'olya urn model}
\kwd{Stein's method}
\kwd{distributional transformations}
\kwd{random walk}
\kwd{random binary trees}
\kwd{random plane trees}
\kwd{preferential attachment random graphs}
\end{keyword}
\end{frontmatter}

\section{Introduction}\label{sec11}

Generalized gamma distributions arise as limits in a variety of
combinatorial settings involving random trees [e.g.,~\citet
{Janson2006a}, \citet{Meir1978} and \citet
{Panholzer2004}], urns [e.g.,~\citet{Janson2006}], and walks
[e.g.,~\citet{Chung1976}, \citet{Chung1949} and
\citet{Durrett1977}]. These distributions are those of gamma
variables raised to a power and noteworthy examples are the Rayleigh
and half-normal distributions. We show that for a family of time
inhomogeneous generalized P\'olya urn models, nearly the full class of
generalized gamma distributions with integer parameters appear as
limiting distributions, and we provide optimal rates of convergence to
these limits. Apart from some special cases, both the characterizations
of the limit distributions and the rates of convergence are new.

The result for our urn model (Theorem~\ref{thm1} below) follows from
a general approximation result
(Theorem~\ref{thm4} below) which provides a framework
for bounding the distance between a generalized gamma
distribution and a distribution of interest. This result is derived
using Stein's method [see \citet{Ross2011},
\citet{Ross2007} and \citet{Chen2011} for overviews]
coupled with characterizing the generalized gamma
distributions as unique fixed points of certain
distributional transformations.
Similar approaches to deriving approximation results have found past
success for other distributions in many applications:
the size-bias transformation for Poisson approximation by \citet
{Barbour1992},
the zero-bias transformation for normal approximation by
\citeauthor{Goldstein1997} (\citeyear{Goldstein1997,Goldstein2005a}) [and a discrete analog of~\citet
{Goldstein2006}],
the equilibrium transformation of renewal theory for both exponential
and geometric approximation, and an
extension to negative binomial approximation by~\citet
{Pekoz2011}, \citet{Pekoz2013a} and \citet{Ross2013},
and a transformation for a class of distributions arising in preferential
attachment graphs by \citet{Pekoz2013}. \citet{Luk1994} and
\citet{Nourdin2009} developed Stein's method for gamma
approximation, though the approaches there are quite different from ours.
Theorem~\ref{thm4} is a significant generalization and embellishment
of this previous work.

Using the construction of \citet{Remy1985} for generating uniform random
binary trees, we find some of our urn distributions embedded in random
subtrees of uniform binary trees and plane trees. Moreover, a
well-known bijection
between binary trees and Dyck paths
yields analogous embeddings in some local time and height
statistics of random walk.
By means of these embeddings, we are able to prove convergence to
generalized gamma distributions
with rates for these statistics. These limits and in general the
connection between random walks, trees
and distributions appearing in Brownian motion are typically
understood through classical bijections between trees and walks along
with Donsker's
invariance principle, or through the approach of Aldous' continuum
random tree; see \citet{Aldous1991}.
While these perspectives are both beautiful and powerful,
the mathematical details are intricate and they do not provide
rates of convergence. In this setting, our work can be viewed as a
simple unified approach to understanding the appearance
of these limits in the tree-walk context which has the added benefit of
providing rates of convergence.

In the remainder of the \hyperref[sec11]{Introduction}, we state our urn, tree and walk
results in detail.

\subsection{Generalized gamma distribution}

For $\alpha>0$, denote by $\Ga(\alpha)$ the gamma distribution with
shape parameter $\alpha$ having
density $x^{\alpha-1} e^{-x}/\G(\alpha)\,dx$, $x>0$.

%
%
\begin{definition}[(Generalized gamma distribution)]
For positive real numbers $\alpha$ and $\beta$, we say a random
variable $Z$ has
the \emph{generalized gamma distribution with parameters} $\alpha$
\emph{and} $\beta$ and write $Z\sim\GG(\alpha,\beta)$,
if $Z\stackrel{\mathscr{D}}{=}X^{1/\beta}$, where $X\sim\Ga
(\alpha/\beta)$.
\end{definition}
The density of $Z\sim\GG(\alpha,\beta)$ is easily seen to be
\[
\varphi_{\alpha,\beta}(x) = \frac{\beta x^{\alpha-1}e^{-x^\beta
}}{\Gamma(\alpha/\beta)} \,dx,\qquad x>0,
\]
and for any real $p>-\alpha$, $\IE Z^p = \Gamma((\alpha+p)/\beta
)/\Gamma(\alpha/\beta)$; in particular $\IE Z^\beta=\alpha/\beta$.
The generalized gamma family includes the Rayleigh distribution, $\GG
(2,2)$, the absolute or ``half'' normal distribution, $\GG(1,2)$,
and the standard gamma~distribution, $\GG(\alpha, 1)$.

\subsection{P\'olya urn with immigration}

We now define a variation of P\'olya's urn. An urn starts with black
and white balls and draws are made sequentially. After each draw, the
ball is replaced and another ball of the same color is added to the urn.
Also, after every $l$th draw an additional black ball is added to the
urn. Let ${\mathcal{P}}^l_n(b,w)$ denote the distribution of the
number of white
balls in the urn after $n$ draws have been made when the urn starts
with $b\geq0$ black balls and $w>0$ white balls. Note that for the
case $l=1$ the process is time homogeneous but for $l\geq2$ it is time
inhomogeneous. Define the \emph{Kolmogorov distance} between two
cumulative distribution functions $P$ and $Q$ (or their respective
laws) as
\[
\dk(P,Q)=\sup_x\bigl\llvert P(x)-Q(x)\bigr\rrvert.
\]
The Kolmogorov metric is a standard and natural metric for random
variables on the real line and
is used for statistical inference, for example, in computing \mbox{``$p$-}values''.

%
%
\begin{theorem}\label{thm1}
Let $l,w\geq1$ and let $\urn_n \sim{\mathcal{P}}^{l}_{n}(1,w)$.
Then $\IE\urn_n^k\asymp n^{kl/(l+1)}$ as $n\to\infty$ for any
integer $k\geq0$, and
\begin{eqnarray*}
&&\IE\urn_n^{l+1}\sim n^l w \biggl(
\frac{l+1}{l} \biggr)^{l}.
\end{eqnarray*}
Furthermore, there are constants $c =c_{l,w}$ and $C =C_{l,w}$,
independent of $n$, such that
%
%
\begin{equation}
\label{1} c n^{-l/(l+1)}\leq\dk\bigl(\law(\urn_{n}/
\mu_{n}), \GG(w,l+1) \bigr) \leq C n^{-l/(l+1)},
\end{equation}
where
%
%
\begin{equation}
\label{2} \mu_n=\mu_n(l, w)= \biggl(
\frac{l+1}{w} \IE\urn_n^{l+1} \biggr)^{1/(l+1)}
\sim n^{l/(l+1)}\frac{(l+1)}{l^{l/(l+1)}}.
\end{equation}
\end{theorem}

%
%
\begin{remark}
A direct application of this result is to a preferential attachment
random graph model [see \citet{Barabasi1999}, \citet{Pekoz2013}]
that initially has one node having weight $w$ (thought of as the degree
of that node or a collection of nodes grouped together).
Additional nodes are added sequentially and when a node is added it
attaches $l$ edges, one at a time, directed
{from} it {to} either itself or to nodes in the existing graph
according to the following rule. Each edge attaches to a potential node
with chance proportional to that node's weight at that exact moment,
where incoming edges contribute weight one to a node and each node
other than the initial node is born having initial weight one.
The case where $l=1$ is the usual Barabasi--Albert tree with loops
(though started from a node with initial weight $w$ and no edges).
A moment's thought shows that after an additional $n$ edges have been
added to the graph,
the total weight of the initial node
has distribution ${\mathcal{P}}^{l}_{n}(1,w)$.
\citet{Pekoz2014} extend the results of this paper in this
preferential attachment
context to obtain limits for joint distributions of the weights of nodes.
\end{remark}

%
%
\begin{remark}\label{rem1}
Theorem~\ref{thm1} in the case when $l=1$ is covered by Example~3.1
of \citet{Janson2006}, but without a rate of convergence. The
limit and rate
for the two special cases where $w=l=1$ and $l=1$, $w=2$ are stated in
Theorem~1.1 of \citet{Pekoz2013}; in fact the rate proved there
is $n^{-1/2}\log n $ (there is an error in the last line of the proof
of their Lemma~4.2), but our
approach here yields the optimal rate claimed there.
\end{remark}

%
%
\begin{remark} For $n\geq l$, it is clear that
%
%
\begin{equation}
\label{3} {\mathcal{P}}^l_n(0,w) = {
\mathcal{P}}^l_{n-l}(1,w+l),
\end{equation}
since, if the urn is started without black balls, the progress of the
urn is deterministic until the first immigration. ${\mathcal
{P}}^l_{n}(1,w)$ is
more natural in the context of the proof of Theorem~\ref{thm1} but in
our combinatorial applications, ${\mathcal{P}}^l_n(0,w)$ can be easier
to work
with and so we will occasionally apply Theorem~\ref{thm1}
directly to ${\mathcal{P}}^l_n(0,w)$ via~\eq{3}. Further, in order to easily
switch between these two cases without introducing unnecessary notation
or case distinctions, we define, in accordance with~\eq{3}, ${\mathcal{P}}
^l_{-i}(1,w+l)$ to be a point mass at $ w+l-i$ for all $0\leq i\leq l$.
\end{remark}

%
%
\begin{remark}\label{rem2}
P\'olya urn schemes have a long history and large literature. In brief,
the basic model, in which the urn starts with $w$ white and $b$ black
balls and at each stage
a ball is drawn at random and replaced with $\alpha$ balls of the same color,
was introduced in~\citet{Eggenberger1923} as a model for
disease contagion. The proportion
of white balls converges almost surely to a variable having beta
distribution with parameters $(w/\alpha, b/\alpha)$.
A well-known embellishment [see \citet{Friedman1949}]
is to replace the ball drawn along with $\alpha$ balls of the same
color and $\beta$ of the other color and here if $\beta\neq0$ the
proportion of white balls
almost surely converges to $1/2$; and \citet{Freedman1965} proves
a Gaussian limit theorem for the fluctuation around this limit.

The general case can be encoded by $(\alpha, \beta; \gamma, \delta
)_{b,w}$ where now the urn starts
with $b$ black and $w$ white balls and at each stage a ball is drawn
and replaced; if the ball drawn is black (white), then $\alpha$
($\gamma$)
black balls and $\beta$ ($\delta$) white balls are added.
As suggested by the previous paragraph, the
limiting behavior of the urn can vary wildly depending on the
relationship of the six parameters involved and especially the
Greek letters; even the first-order growth of the number of white balls
is highly
sensitive to the parameters.

A useful tool for analyzing the general
case is to embed the urn process into a multitype branching process
and use the powerful theory available there. This was first
suggested and implemented by \citet{Athreya1968} and has found subsequent
success in many further works; see \citet{Janson2006} and
\citet{Pemantle2007}, and references therein.
An alternative approach that is especially useful when $\alpha$
or $\delta$ are negative
(under certain conditions this leads to a \emph{tenable} urn)
is the analytic combinatorics methods of \citet{Flajolet2005};
see also the \hyperref[sec11]{Introduction} there for further
references.

Note that all of the references of the previous paragraphs
regard homogeneous urn processes and so do not
directly apply to the model of Theorem~\ref{thm1} with $l\geq2$.
In fact, the extensive survey \citet{Pemantle2007} has only a
small section with
a few references regarding time dependent urn models.
Time inhomogeneous urn models do have an extensive statistical
literature due to the their wide usage in the experimental design of
clinical trials
(the idea being that it is ethical to favor experimental treatments
that initially do well
over those that initially do not); see \citet{Zhang2006},
\citet{Zhang2011} and \citet{Bai2002}.
This literature is concerned with models and regimes where
the asymptotic behavior is Gaussian. As discussed in \citet{Janson2006},
it is difficult to characterize nonnormal asymptotic
distributions of generalized P\'olya urns, even in the time homogeneous case.
\end{remark}

%
%
\begin{remark}
There are many possible natural generalizations of the model we study
here, such as
starting with more than one black ball or adding more than one black
ball every $l$th draw.
We have restricted our study to the ${\mathcal{P}}_n^{l}(1,w)$ urn
because these
variations
lead to asymptotic distributions outside the generalized gamma class.
For example, the case ${\mathcal{P}}^1_n(b,w)$ with integer $b\geq1$
is studied
in \citet{Pekoz2013},
where it is shown for $b\geq2$
the limits are powers of products of independent beta and gamma random
variables.
Our main purpose here
is to study the generalized gamma regime carefully and to highlight the
connection between these urn models and random walks and trees.
\end{remark}

\subsection{Applications to sub-tree sizes in uniform binary and plane trees}

Denote by $T^p_n$ a uniformly chosen rooted plane
tree with $n$ nodes, and denote by $T^b_{2n-1}$ a uniformly chosen
binary, rooted plane tree with $2n-1$ nodes, that is, with $n$ leaves
and $n-1$ internal nodes.
It is well known that the number of such trees in both cases is the
Catalan number $C_{n-1}={2n-2\choose n-1}/n$
and that both families of random trees are instances of simply
generated trees; see Examples~10.1 and~10.3 of \citet{janson2012}.

For any rooted tree $T$ let $\operatorname{sp}
^{k}_{\mathrm{Leaf}}(T)$ be the number of
vertices in the minimal spanning tree spanned by the root and $k$
randomly chosen distinct \emph{leaves} of $T$, and let $
\operatorname{sp}^{k}_{\mathrm{Node}}(T)$ be the number of
vertices in the minimal spanning tree
spanned by the root and $k$ randomly chosen distinct \emph{nodes} of $T$.

%
%
\begin{theorem}\label{thm2} Let $\mu_n(1,w)$ be as in \eq{2} of
Theorem~\ref{thm1}.
Then, for any $k\geq1$,
\begin{eqnarray*}
\mbox{(\textup{i})}&\quad&  \dk\bigl(\law\bigl({\operatorname{sp}
^{k}_{\mathrm{Leaf}}\bigl(T^b_{2n-1}\bigr)}/{
\mu_{n-k-1}(1,2k)} \bigr),\GG(2k,2) \bigr) = \bigo\bigl(n^{-1/2}
\bigr),
\\
\mbox{(\textup{ii})}&\quad&  \dk\bigl(\law\bigl({\operatorname{sp}
^{k}_{\mathrm{Node}}\bigl(T^b_{2n-1}\bigr)}/{
\mu_{n-k-1}(1,2k)} \bigr),\GG(2k,2) \bigr) = \bigo\bigl(n^{-1/2}
\bigr),
\\
\mbox{(\textup{iii})}&\quad & \dk\bigl(\law\bigl({2\operatorname{sp}
^{k}_{\mathrm{Node}}
\bigl(T^p_{n}\bigr)}/{\mu_{n-k-1}(1,2k)} \bigr),
\GG(2k,2) \bigr) = \bigo\bigl(n^{-1/2}\log n \bigr).
\end{eqnarray*}
\end{theorem}

%
%
\begin{remark}
The logarithms in (iii) of the theorem and in (iii) and (iv) of
the forthcoming Theorem~\ref{thm3}
are likely an artifact of our analysis, specifically in the use of
Lemma~\ref{lem2}.
\end{remark}

%
%
\begin{remark}\label{rem3}
The limits in the theorem can also be seen using facts about the
Brownian continuum random tree (CRT) due to \citeauthor{Aldous1991} (\citeyear{Aldous1991,Aldous1993}).
Indeed, the trees $T_{2n-1}^b$ and $T_n^p$ can\vspace*{2pt} be understood to
converge in a certain sense to the Brownian CRT. The limit of the
subtrees we study having $k$ leaves can be defined through the
Poisson line-breaking construction as described following Theorem~7.9
of \citet{Pitman2006}:
\[
\begin{tabular}{p{300pt}}
Let $0<\Theta_1<
\Theta_2<\cdots$ be the points of an inhomogeneous Poisson process
on $\IR_{>0}$ of rate $t \,dt$. Break the line $[0,\infty)$ at
points $\Theta_k$. Grow trees $\mathcal{T}_k$ by
letting $\mathcal{T}_1$ be a segment of length $\Theta_1$,
then for $k\geq2$ attaching the segment $(\Theta_{k-1},
\Theta_k]$ as a ``twig'' attached at
a random point of the tree $\mathcal{T}_{k-1}$ formed from the
first $k-1$ segments.
\end{tabular}
\]
The length of this tree is just $\Theta_k$ which is the generalized
gamma limit
of the theorem (up to a constant scaling). In\vspace*{1pt} more detail, if we
jointly generate the vector
$\mathbf{U}_k(n):=(\operatorname{sp}^{1}_{\mathrm
{Leaf}}(T^b_{2n-1}),\ldots,\operatorname{sp}
^{k}_{\mathrm{Leaf}}(T^b_{2n-1}))$ by first selecting $k$
leaves uniformly at random from $T^b_{2n-1}$, then labeling the
selected leaves $1,\dots,k$, and then
setting $\operatorname{sp}^{i}_{\mathrm
{Leaf}}(T^b_{2n-1})$ to be the number of nodes in
the tree spanned by the root and the leaves labeled $1,\dots,i$, then
the CRT theory implies $n^{-1/2}\mathbf{U}_k(n)$ converges in
distribution to $(\Theta_1,\ldots,\Theta_k)$; see also \citet
{Pekoz2014} for
a proof of this fact with a rate of convergence.

\citeauthor{Panholzer2004} [(\citeyear{Panholzer2004}), Theorem~6] provides local
limit theorems for\break
$\operatorname{sp}^{k}_{\mathrm{Node}}(T^b_{2n-1})$
and $\operatorname{sp}^{k}_{\mathrm{Node}}(T^p_{n})$,
from which the distributional
convergence to the generalized gamma can be seen. It may be possible to
obtain such (and other) local limits results
using our Kolmogorov bounds and the approach of \citet
{Rollin2012a}, but
in any case the convergence rates in the Kolmogorov metric in
Theorem~\ref{thm2} appear to be new.
\end{remark}

\subsection{Applications to occupation times and heights in random
walk, bridge and meander}

Consider the one-dimensional simple symmetric random walk $S_n =
(S_n(0),\dots,S_n(n))$ of length $n$ starting at the origin. Define
\[
L_n = \sum_{i=0}^{n} \I
\bigl[S_n (i) =0\bigr]
\]
to be the number of times the random walk visits the origin by
time $n$. Let
\[
L^b_{2n} \sim\law\bigl(L_{2n} \bmid
S_{2n}(0)= S_{2n}(2n)=0 \bigr)
\]
be the local time of a random walk bridge, and define random walk
excursion and meander by
\begin{eqnarray*}
S_{2n}^e & \sim&\law\bigl(S_{2n} \bmid
S_{2n}(0)=0, S_{2n}(1)>0,\dots,S_{2n}(2n-1) > 0,
S_{2n}(2n)=0 \bigr),
\\
S_n^m & \sim&\law\bigl(S_n \bmid
S_{n}(0)= 0, S_{n}(1)>0,\dots,S_{n}(n) > 0
\bigr).
\end{eqnarray*}

%
%
\begin{theorem}\label{thm3}
Let $\mu_n(1,w) = \mu_n$ be as in \eq{2} of Theorem~\ref{thm1} and let
$K$ be uniformly distributed on $\{0,\dots,2n\}$ and independent of
$(S^e_{2n}(u))_{u=0}^{2n}$. Then
\begin{eqnarray*}
\mbox{(\textup{i})} &\quad& \dk\bigl(\law\bigl({ L_{n} }/{\mu_{{\lfloor
n/2\rfloor}}(1,1)}
\bigr),\GG(1,2) \bigr) = \bigo\bigl(n^{-1/2} \bigr),
\\
\mbox{(\textup{ii})} &\quad& \dk\bigl(\law\bigl({ L^b_{2n} }/{\mu
_{n-1}(1,2)} \bigr),\GG(2,2) \bigr) = \bigo\bigl(n^{-1/2}
\bigr),
\\
\mbox{(\textup{iii})} &\quad& \dk\bigl(\law\bigl({2S^e_{2n}(K)}/{\mu
_{n-2}(1,2)} \bigr),\GG(2,2) \bigr) = \bigo\bigl(n^{-1/2}\log
n \bigr),
\\
\mbox{(\textup{iv})} &\quad& \dk\bigl(\law\bigl({ 2S^m_{n} }/{
\mu_{{\lfloor
(n-1)/2\rfloor}-1}(1,2)} \bigr),\GG(2,2) \bigr) = \bigo\bigl(n^{-1/2}
\log n \bigr).
\end{eqnarray*}
\end{theorem}

%
%
\begin{remark}\label{rem4}
An alternative viewpoint of the limits in Theorem~\ref{thm3}
is that they are the analogous
statistics of Brownian motion, bridge, meander and excursion which can
be read from \citet{Chung1976} and \citet{Durrett1977};
these Brownian fragments are the
weak limits in the path space $C[0,1]$ of the walk fragments we study;
see \citet{Csaki1981}.
For example, if $B_t$, $t\geq0$, is a standard Brownian
motion and $(L_t^x, t\geq0,x\in\IR)$ its local time at level $x$ up
to time $t$,
then L\'evy's identity implies that $L_1^0$ is equal in distribution to
the maximum of $B_t$ up to time $1$, which is equal in distribution to
a half normal distribution; see also \citet{Borodin1987}.

To check the remaining limits of the theorem [which are Rayleigh, $\GG
(2,2)$], we can use \citet{Pitman1999b}, equation~(1) [see also
\citet
{Borodin1989}] which states that for $y>0$ and $b\in\IR$,
%
%
\begin{eqnarray} \label{4}
\quad && \IP\bigl[L_1^x\in dy, B_1\in db\bigr]
\nonumber\\[-8pt]\\[-8pt]\nonumber
&&\qquad =
\frac{1}{\sqrt{2\pi}} \bigl(\llvert x\rrvert+\llvert b-x\rrvert+y
\bigr) \exp
\biggl(-\frac{1}{2}\bigl(\llvert x\rrvert+\llvert b-x\rrvert+y
\bigr)^2 \biggr) \,dy \,db.
\end{eqnarray}
Roughly, for the local time of Brownian bridge at time $1$ we
set $b=x=0$ in~\eq{4} and multiply by $\sqrt{2\pi}$ (due to
conditioning $B_1=0$)
to see the Rayleigh density.
For the final time of Brownian meander, we set $x=y=0$ in~\eq{4} and
multiply by $\sqrt{\pi/2}$ (due to conditioning $L_1^0=0$),
and note here that $b\in\IR$ so by symmetry we restrict $b>0$ and
multiply by $2$ to get back to the Rayleigh density.
Finally, due to Vervaat's transformation [\citet{Vervaat1979}],
the height of standard Brownian excursion at a uniform random time has
the same distribution as the maximum of Brownian bridge on $[0,1]$. If
we denote by $M$ this maximum, then for $x>0$
we apply~\eq{4} to obtain
\begin{eqnarray*}
\IP[M>x]&=&\IP\bigl[L_1^x>0\mid B_1=0\bigr] =
\int_0^\infty(2x+y) \exp\biggl(-\frac{1}{2}
(2x+y )^2 \biggr) \,dy =e ^{-2x^2},
\end{eqnarray*}
which is the claimed Rayleigh distribution.

With the exception of the result for $L_n$, which can be read from
\citet{Chung1949}, inequality~(1) or \citet{Dobler2013}, Theorem~1.2, the convergence rates
appear to be new.
\end{remark}

\subsection{A general approximation result via distributional transforms}

Theorem~\ref{thm1} follows from a general approximation result using
Stein's method, a distributional transformation
with a corresponding fixed point equation, which we describe now. We
first generalize the size
bias transformation used in Stein's method and appearing naturally in
many places; see~\citet{Arratia2013} and \citet{Brown2006}.

%
%
\begin{definition}
Let $\beta>0$ and let $W$ be a nonnegative random variable with
finite $\beta$th moment. We say a random
variable $W^{(\beta)}$ has the~\emph{$\beta$-power bias distribution
of $W$}, if
%
%
\begin{equation}
\label{5} \IE\bigl\{W^\beta f(W) \bigr\} =\IE W^\beta\IE
f \bigl(W^{(\beta
)} \bigr)
\end{equation}
for all $f$ for which the expectations exist.
\end{definition}

In what follows, denote by $\B(a,b)$ the beta distribution with
parameters \mbox{$a,b>0$}.

%
%
\begin{definition}\label{def1}
Let $\alpha>0$ and $\beta>0$ and let $W$ be a positive random
variable with $\IE W^\beta= \alpha/\beta$. We\vspace*{2pt} say that $W^{*}$
has the \emph{$(\alpha,\beta)$-generalized equilibrium distribution
of $W$}\/ if,
for $V_\alpha\sim\B(\alpha,1)$ independent of $W^{(\beta)}$, we have
%
%
\begin{equation}
\label{6} W^*\ed V_\alpha W^{(\beta)}.
\end{equation}
\end{definition}

%
%
\begin{remark}\label{7}
\citet{Pakes1992}, Theorem~5.1 and \citet{Pitman2012}, Proposition~9 show that for a
positive random
variable $W$ with $\IE W^\beta=\alpha/\beta$, we have $W\sim\GG
(\alpha,\beta)$ if and only if $W\ed
W^*$. The $(1,2)$-generalized equilibrium distributional transformation
is the nonnegative
analog of the zero bias transformation of which the standard normal
distributions are unique fixed points;
see \citet{Chen2011}, Proposition~2.3, page~35, where
the $2$-power
bias transformation
is appropriately called ``square'' biasing; thus $\GG(1,2)$ is the
absolute normal distribution.
\end{remark}

%

%
%
\begin{theorem}\label{thm4} Let $W$ be a positive random variable
with $\IE W^\beta= \alpha/\beta$ for some integers $\alpha\geq1$
and $\beta\geq1$. Let $W^*$ be a
random variable constructed on the same probability space having
the $(\alpha,\beta)$-generalized equilibrium distribution of~$W$.
Then there is a constant $c>0$ depending only on $\alpha$ and $\beta$
such that, for all $0<b\leq1$,
%
%
\begin{equation}
\label{8} \dk\bigl(\law(W), \GG(\alpha,\beta) \bigr) \leq c \bigl(b+
\IP
\bigl[\bigl\llvert W-W^*\bigr\rrvert>b\bigr] \bigr).
\end{equation}
\end{theorem}

%
%
\begin{remark} Let $X$ and $Y$ be two random variables and let
\[
\dlp\bigl(\law(X),\law(Y)\bigr) = \inf\bigl\{b: \IP[X\leq t]\leq\IP
[Y\leq
t+b] + b\mbox{ for all }t\in\IR\bigr\}
\]
be the L\'evy--Prokohorov distance between $\law(X)$ and $\law(Y)$. A
theorem due to Strassen [see, e.g., \citet{Dudley1968}, Theorem~2] says that there is a
coupling $(X,Y)$ such
that $\IP[\llvert X-Y\rrvert >\rho]\leq\rho$, where $\rho
= \dlp(\law
(X),\law(Y))$. Hence, since \eq{8} holds for \emph{all} $b$ and
\emph{all} couplings of $W$ and $W^*$, it follows in particular that
\[
\dk\bigl(\law(W), \GG(k,r) \bigr) \leq2c \dlp\bigl(\law(W),\law\bigl
(W^*\bigr)
\bigr).
\]
\end{remark}

The paper is organized as follows. In Section~\ref{sec1}, we
embed our urn model into random trees via R\'emy's algorithm
and prove Theorem~\ref{thm2}. In Section~\ref{sec2a}, we describe
the various connections between trees and walk paths and then
prove Theorem~\ref{thm3}. In Section~\ref{sec2}, we use Theorem~\ref{thm4}
to prove Theorem~\ref{thm1}, and finally in Section~\ref{sec3} we
develop a general
formulation of Stein's method for log concave densities and prove
Theorem~\ref{thm4}.

\section{Random trees: Proof of Theorem~\texorpdfstring{\protect\ref{thm2}}{1.8}}
\label{sec1}



\makeatletter
\def\drawdot#1{\draw(#1) node [draw,circle,fill=black,inner
sep=0.5pt] {}}

\tikzset{
int/.style = {circle, white, font=\bfseries,
draw=black, fill=black,inner sep=0.5pt},
ext/.style = {circle, draw, inner sep=0.5pt},
arr/.style = {shorten >=1mm,shorten <=1mm},
arrd/.style = {arr,densely dotted},
arrn/.style = {sloped,pos=0.5}
}
\def\@intnode#1{\begin{tikzpicture}\tiny
\node at (0,0) [int,inner sep=1pt,minimum size=3.5ex] {#1};
\end{tikzpicture}}
\def\intnode#1{\protect\raisebox{-0.3ex}{\protect\@intnode#1}}
\def\@extnode#1{\protect{\begin{tikzpicture}\tiny
\node at (0,0) [ext,inner sep=1pt,minimum size=3.5ex] {#1};
\end{tikzpicture}}}
\def\extnode#1{\protect\raisebox{-0.3ex}{\protect\@extnode#1}}
\def\pdot{\phantom{$\cdot$}}
\def\fixbox#1{\hbox to 1.1em {\hfil #1\hfil}}
\makeatother

\subsection{R\'emy's algorithm for decorated binary trees}

R{\'e}my (\citeauthor{Remy1985}) introduced an elegant recursive algorithm to
construct uniformly chosen decorated binary trees, where by ``decorated''
we mean that the leaves are labeled. This algorithm is the key
ingredient to our
approach as it relates to the urn schemes of Theorem~\ref{thm1}. All
trees are assumed to be plane trees throughout, and we will think of
the tree as growing
downward with the root at the top. We will refer to the ``left'' and
``right'' child of a node as seen from the readers point of view looking
at the tree growing downward.

\subsubsection*{R\'emy's algorithm for decorated binary trees (see Figure~\protect\ref{fig1})}
Let $n\geq1$ and assume
that $T^b_{2n-1}$ is a uniformly chosen decorated binary tree with $n$
leaves, labeled from $1$ to $n$. To obtain a uniformly chosen
decorated binary tree $T^b_{2n+1}$ with $n+1$ leaves do the following:
\begin{longlist}
\item[\textit{Step}~1.] Choose a node uniformly at random; call it $X$.
Remove $X$ and its sub-tree, insert a new internal node at this
position, call it $Y$, and attach $X$ and its sub-tree,
to $Y$.
\item[\textit{Step}~2.] With probability $1/2$ each, do either of the
following:
\begin{longlist}[(a)]
\item[(a)] Attach new leaf
with label $n+1$ as the left-child to $Y$ (making $X$ the right-child
of $Y$).

\item[(b)] Attach new leaf
with label $n+1$ as the right-child to $Y$ (making $X$ the left-child
of $Y$).
\end{longlist}
\end{longlist}

%
%
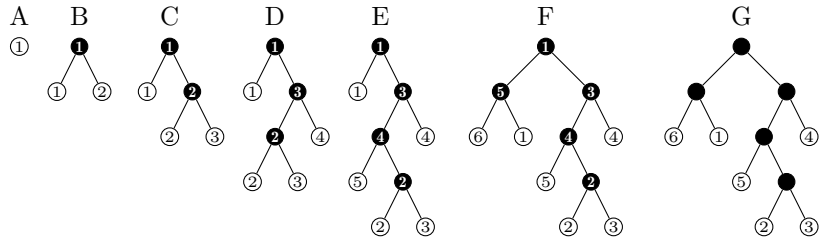
\begin{figure}[b]
\hfill
\begin{tikzpicture}[scale=0.4]\tiny
\node at (0,1) {\small A};
\node at (0,0) [ext] {1}
;
\node at (2,1) {\small B};
\node at (2,0) [int] {1}
child { node [ext] {1} }
child { node [ext] {2} }
;
\node at (5,1) {\small C};
\node at (5,0) [int] {1}
child { node [ext] {1} }
child { node [int] {2}
child { node [ext] {2} }
child { node [ext] {3} }
}
;
\node at (8.5,1) {\small D};
\node at (8.5,0) [int] {1}
child { node [ext] {1} }
child { node [int] {3}
child { node [int] {2}
child { node [ext] {2} }
child { node [ext] {3} }
}
child { node [ext] {4} }
}
;
\node at (12,1) {\small E};
\node at (12,0) [int] {1}
child { node [ext] {1} }
child { node [int] {3}
child { node [int] {4}
child { node [ext] {5} }
child { node [int] {2}
child { node [ext] {2} }
child { node [ext] {3} }
}
}
child { node [ext] {4} }
}
;
\tikzstyle{level 1}=[sibling distance=3cm]
\tikzstyle{level 2}=[sibling distance=1.5cm]
\node at (17.5,1) {\small F};
\node at (17.5,0) [int] {1}
child { node [int] {5}
child { node [ext] {6} }
child { node [ext] {1} }
}
child { node [int] {3}
child { node [int] {4}
child { node [ext] {5} }
child { node [int] {2}
child { node [ext] {2} }
child { node [ext] {3} }
}
}
child { node [ext] {4} }
}
;
\tikzstyle{level 1}=[sibling distance=3cm]
\tikzstyle{level 2}=[sibling distance=1.5cm]
\node at (24,1) {\small G};
\node at (24,0) [int] {\phantom{1}}
child { node [int] {\phantom{1}}
child { node [ext] {6} }
child { node [ext] {1} }
}
child { node [int] {\phantom{1}}
child { node [int] {\phantom{1}}
child { node [ext] {5} }
child { node [int] {\phantom{1}}
child { node [ext] {2} }
child { node [ext] {3} }
}
}
child { node [ext] {4} }
}
;
\end{tikzpicture}
\hfill\hfill
\caption{Illustration of R\'emy's algorithm to construct
decorated binary trees. Internal nodes are represented by
black circles, and leaves by white circles. For the sake of clarity,
we keep the internal nodes labeled, but these labels will be removed
in the
final step. We start with Tree~A, the trivial tree. The step from
Tree~A to Tree~B is
\extnode{1}L, where ``\extnode{1}'' indicates the node that was chosen,
and ``L'' indicates that this node, along with its sub-tree, is attached
to the new node as the left-child. Using
this notation, the remaining steps to get to Tree~F are \extnode{2}L,
\intnode{2}L, \intnode{2}R, \extnode{1}R. Then remove the labels of the
internal nodes to obtain Tree~G, the final tree.}
\label{fig1}
\end{figure}

This recursive algorithm produces uniformly chosen decorated binary trees,
since every decorated binary tree can be obtained in exactly one way,
and since at every iteration every new tree is chosen with equal
probability. By removing the labels, we obtain a uniformly chosen
undecorated binary tree.

Figure~\ref{fig1} illustrates the algorithm by means of an example.
We have labeled the internal nodes to make the procedure clearer, but
it is
important to note that these internal labels are not chosen uniformly among
all such labelings and, therefore, have to be removed at the final
step (to see
this, note that Tree~C in Figure~\ref{fig1} cannot be obtained through
R\'emy's algorithm if the labels of the two internal nodes are switched).

\subsection{Sub-tree sizes}

\subsubsection*{Spanning trees in binary trees} R\'emy's algorithm creates
a direct embedding of a P\'olya urn into a decorated binary tree.
The following result is the key to our tree and walk results and is
utilized via embeddings and bijections in this section and the following.
The result is implicit
in a construction of \citet{Pitman2006}, Exercise 7.4.11.

%
%
\begin{proposition}\label{prop1} For any $n\geq k\geq1$,
\[
\operatorname{sp}^{k}_{\mathrm{Leaf}}
\bigl(T^b_{2n-1}\bigr) \sim{\mathcal{P}}^{1}_{n-k}(0,2k-1)
= {\mathcal{P}} ^{1}_{n-k-1}(1,2k).
\]
\end{proposition}

\begin{pf} Since the labeling is random, we may consider
the tree spanned by the root and the leaves labeled $1$ to $k$ of a
uniformly chosen decorated binary tree,
rather than the tree spanned by the root and $k$ uniformly chosen
leaves of a random binary tree, cf. \citet{Pitman2006}, Exercise~7.4.11.
Start with a uniformly chosen decorated binary tree $T^b_{2k-1}$
with $k$ leaves and note that the tree spanned
by the root and leaves $1$ to $k$ is the whole tree. Now identify
the $2k-1$ nodes of $T^b_{2k-1}$ with $2k-1$ white balls in an urn that
has no black balls.
If the randomly chosen node in a given step in R\'emy's algorithm is
outside the current spanning tree, two nodes will be added
outside the current spanning tree and we identify this as adding two
black balls to the urn. If the randomly chosen node is in the current
spanning tree, one node will be added to the current spanning tree and
another outside of it, and we identify this as adding one black and one
white ball to the urn.

Since we started with a tree of $2k-1$ nodes, we need $n-k$ steps to
obtain a tree with $2n-1$ nodes.
Hence, the size of the spanning tree is equal to the number of white
balls in the urn, which follows the distribution ${\mathcal
{P}}^{1}_{n-k}(0,2k-1)$.
\end{pf}

As a consequence of Proposition~\ref{prop1}, whenever a quantity of
interest can be coupled closely to $\operatorname{sp}
^{k}_{\mathrm{Leaf}}(T^b_{2n-1})$,
rates of convergence can be obtained if the closeness of the coupling
can be quantified appropriately. In this section, we give two tree
examples of this approach. Since the distribution ${\mathcal{P}}
^{1}_{n-k}(0,2k-1)$ will appear over and over again, we set $N^*_{j,k}
\sim{\mathcal{P}}^{1}_{j}(0,2k-1)$
in what follows. We use\vspace*{1pt} the notation $\Gez(p)$, $\Geo(p)$, $\Be(p)$,
$\Bi(n,p)$ to, respectively, denote the geometric with supports starting
at zero and one, Bernoulli and binomial distributions. For a
nonnegative integer-valued random variable $N$, we also use the
notation $X\sim\Bi(N,p)$ to denote that $X$ is distributed as a
mixture of binomial distributions such that $\law(X\mid N=n)=\Bi(n,p)$.

We now make a simple, but important observation about the edges in the
spanning tree.

%
%
\begin{lemma}\label{lem1} Let $1\leq k\leq n$ and $T^b_{2n-1}$ be a
uniformly chosen binary tree with $n$ leaves and consider the tree
spanned by the root and $k$ uniformly chosen distinct leaves.
Let $M_{k,n}$ be the number of edges in this spanning tree that connect
a node to its left-child (\hspace*{-1pt}``left-edges''). Conditional on the spanning
tree having $N_{n-k,k}^*$ nodes,
%
%
\begin{equation}
\label{9} M_{k,n} - (k-1) \sim\Bi\bigl(N_{n-k,k}^*-(2k-1),1/2
\bigr).
\end{equation}
\end{lemma}
\begin{pf} We use R\'emy's algorithm and induction over $n$. Fix $k\geq
1$. For $n=k$ note that the spanning tree is the whole tree
with $N^*_{0,k}=2k-1$ nodes and $2(k-1)$ edges. Since half of the edges
must connect a node to the left-child, $M_{k,k} = k-1$ which is~\eq{9}
and this proves the base case. Assume now that~\eq{9} is true for
some $n\geq k$. Two things can happen when applying R\'emy's algorithm:
either the current spanning tree is not changed, in which
case $N^*_{n-k+1,k} = N^*_{n-k,k}$ and $M_{k,n+1} = M_{k,n}$, and
hence~\eq{9} holds by the induction hypothesis, or one node and one
edge are inserted into the spanning tree, in which case $N^*_{n-k+1,k}
= N^*_{n-k,k} + 1$ and $M_{k,n+1} = M_{k,n} + J$ with $J\sim\Be(1/2)$
independent of all else. In the latter case, using the induction
hypothesis, $M_{k,n} + J -(k-1) \sim\Bi(N^*_{n-k,k}-(2k-1)+1,1/2) =
\Bi(N^*_{n-k+1,k}-(2k-1),1/2)$,
which is again~\eq{9}. This concludes the induction step.
\end{pf}

%
%
\begin{proposition}\label{prop2} Let $n\geq k\geq1$ and
let $N^*_{j,k}\sim{\mathcal{P}}^{1}_{j}(0,2k-1)$. There exist nonnegative,
integer-valued random variables $Y_1,\dots,Y_k$ such that, for each $i$,
%
%
\begin{equation}
\IP\bigl[Y_i > m\mid N^*_{n-k,k}\bigr] \leq2^{-m}
\qquad\mbox{for all }m\geq0,\label{10}
\end{equation}
and such that for
%
%
\begin{equation}
\label{12} X_{n,k}:= N^*_{n-k,k} - \sum
_{i=1}^k Y_k
\end{equation}
we have
%
%
\begin{equation}
\label{11} \dtv\bigl(\law\bigl(\operatorname{sp}
^{k}_{\mathrm
{Node}}
\bigl(T^b_{2n-1}\bigr) \bigr),\law(X_{n,k}) \bigr)
\leq\frac{k}{2n}+\frac{(k-1)^2}{2n-k+1},
\end{equation}
where $\dtv$ denotes total variation distance.

For $k=1$ we have the explicit representation
%
%
\begin{equation}
\label{13} \law\bigl(\operatorname{sp}^{1}_{\mathrm
{Node}}
\bigl(T^b_{2n-1}\bigr) \bigr) = \law(X_{n,1} \mid
X_{n,1} > 0),
\end{equation}
where $Y_1\sim\Gez(1/2)$ is independent of $N^*_{n-1,1}$.
\end{proposition}

\begin{pf} We first prove~\eq{13}. We start by regarding $T^b_{2n-1}$
as being ``planted'', that is, we think of the root node as being the
left-child of a ``ground node'' (which itself has no right-child). We
also think of the ground node as being internal. Furthermore, we think
of the minimal spanning tree between the ground node and the root node
as being empty, hence its size as being $0$. We first construct a
pairing between leaves and internal nodes as follows (see Figure~\ref
{fig2}). Pick a leaf and follow the path from that leaf toward the
ground node and pair the leaf with the first parent of a left-child
encountered in that path. Equivalently, pick an internal node and, in
direction away from the ground node, first follow the left child of
that internal node, and then keep following the right child until
reaching a leaf. In particular, with this algorithm, if a selected leaf
is a left-child it is assigned directly to its parent and the
right-most leaf is assigned to the ground node. The fact that this
description is indeed a pairing follows inductively by considering the
left and right subtrees connected to the root, whereby the left subtree
uses the root of the tree as its ground node.

%
%
\begin{figure}
\def\up{{\tiny$+1$}}
\def\down{{\tiny$-1$}}
\hfill
\begin{tikzpicture}[scale=0.4]\tiny
\tikzstyle{level 1}=[sibling distance=4.5cm, dashed]
\tikzstyle{level 2}=[sibling distance=4.5cm, solid]
\tikzstyle{level 3}=[sibling distance=2.5cm]
\node at (-1,-2) [circle,draw=black,densely dashed,inner sep=0.5pt]
(I0) {\phantom0}
child { node at (-2,0.5) [int,solid] (I1) {\phantom1}
child { node [int] (I5) {\phantom5}
child { node [int] (I6) {\phantom6}
child { node [ext] (E7) { \phantom7 } }
child { node [ext] (E6) { \phantom6 } }
}
child { node [ext] (E1) {\phantom1} }
}
child { node [int] (I3) {\phantom3}
child { node [int] (I4) {\phantom4}
child { node [ext] (E5) {\phantom5} }
child { node [int] (I2) {\phantom2}
child { node [ext] (E2) {\phantom2} }
child { node [ext] (E3) {\phantom3} }
}
}
child { node [ext] (E4) {\phantom4} }
}
};
\draw[latex-latex,arr] (E7).. controls +(up:10mm) and +(190:10mm).. (I6);
\draw[latex-latex,arr] (E6).. controls +(115:15mm) and +(245:15mm).. (I5);
\draw[latex-latex,arr] (E1).. controls +(110:14mm) and +(230:14mm).. (I1);
\draw[latex-latex,arr] (E5).. controls +(up:10mm) and +(190:10mm).. (I4);
\draw[latex-latex,arr] (E2).. controls +(up:10mm) and +(190:10mm).. (I2);
\draw[latex-latex,arr] (E3).. controls +(117:30mm) and +(250:15mm).. (I3);
\draw[latex-latex,arr] (E4).. controls +(100:20mm) and +(240:15mm).. (I0);
\end{tikzpicture}
\hfill\hfill
\caption{Pairing up leaves and internal nodes in planted binary trees.
Note that the right-most leaf in the tree is paired up with the ``ground node''.}
\label{fig2}
\end{figure}

Recall that we are considering the case $k=1$. Now, instead of choosing
a node uniformly at random among the $2n$ nodes of the planted tree
(the ground node included), we may equivalently choose Leaf~1 with
probability $1/2$, or choose the internal node paired with Leaf~1 with
probability $1/2$. Denote by $X_n$ the number of nodes in the path from
the chosen node to the root, denote by $J$ the indicator of the event
that we choose an internal node, and denote by $N^*_{n-1,1}$ the number
of nodes in the path from Leaf~1 up to the root. From Proposition~\ref
{prop1} with $k=1$, we have that $N^*_{n-1,1}\sim{\mathcal{P}}^1_{n-1}(0,1)$.
If $J_1=0$, then $X_{n,1} = N^*_{n-1,1}$. If $J_1=1$, the number of
nodes in the path to the root is that of Leaf~1 minus the number of
nodes until the first parent of a left-child in the path is
encountered. Considering Lemma~\ref{lem1}, given $N^*_{n-1,1}$, the
number of left-edges are $N^*_{n-1,1}$ independent coin tosses with
success probability $1/2$, hence, if $\tilde Y_1$ is the time until the
first parent of a left-child is encountered, we have $\tilde Y_1\sim
\Geo
(1/2)$, truncated at $N^*_{n-1,1}$. Thus, if $J_1=1$, we have $X_{n,1}
= N^*_{n-1,1} - \tilde Y_1\wedge N^*_{n-1,1}$. Putting the two cases
together we obtain the representation $X_n = N^*_{n-1,1} - (J_1\tilde Y_1)
\wedge N^*_{n-1,1}$, which has the same distribution as
%
%
\begin{equation}
\label{14} N^*_{n-1,1} - Y_1\wedge N^*_{n-1,1},
\end{equation}
since $J_1\tilde Y_1\sim\Gez(1/2)$. As $X_{n,1}$ is zero if and only if
the ground node was paired with Leaf~1 (i.e., Leaf~1 being the right
most leaf) and $J_1=1$, conditioning on $X_{n,1}$ being positive is
equivalent to conditioning on choosing any node apart from the ground
node, which concludes~\eq{13}.

Now, let $k$ be arbitrary. In a first step, instead of choosing $k$
distinct nodes at random, choose $k$ distinct leaves at random and, for
each leaf, toss a fair coin $J_i$, $i=1,\dots,k$, to determine whether
to choose the leaf or its internal partner, similar to the
case $k=1$. Denote by $N^*_{n-k,k}$ the number of nodes in the minimal
spanning tree spanned by Leaves~1 to $k$ and the root, and denote
by $X_{n,k}$ the number of nodes in the minimal spanning tree spanned
by the leaves or paired nodes and the root (if one of the chosen nodes
is the ground node, then ignore that node in determining the minimal
spanning tree). It is easy to see through two coupling arguments that
choosing the nodes in this different way introduces a total variation
error of at most
\[
\left[1-\pmatrix{2n-1\cr  k}\Big/ \pmatrix{2n\cr k} \right] + \Biggl[1-\prod
_{i=1}^{k-1} \biggl(1-\frac{i}{2n-i} \biggr)
\Biggr];
\]
the first term stems from the possibility of choosing the ground node,
and the second term from restricting the $k$ nodes to be from different
pairings. From this, \eq{11} easily follows.

It remains to show~\eq{10} and~\eq{12}. For~\eq{12}, for
each $i=1,\dots,k$, let $N^{(i)}_{n-1}$ be the number of nodes in the
path from leaf $i$ up to the root, and let $Y'_i=J_i\tilde Y_i$ be the
geometric random variable from the representation~\eq{14}. With $Y_i =
Y'_i \wedge N^{(i)}_{n-1}$, we hence have
\[
X_{n,k} = N^*_{n-k,k} - \sum_{i=1}^k
Y'_i \wedge N^{(i)}_{n-1} =
N^*_{n-k,k} - \sum_{i=1}^k
Y_i.
\]
It is not difficult to check that $Y_i$ and $N^*_{n-k,k} -
N_{n-1}^{(i)}$ are conditionally independent given $N_{n-1}^{(i)}$.
For~\eq{10}, notice that $\IP[Y_i>m\bmid N_{n-1}^{(i)} ] =
\I
[m < N_{n-1}^{(i)} ]2^{-m}$. Hence,
\begin{eqnarray*}
\IP\bigl[Y_i>m\bmid N_{n-k,k}^* \bigr] & =& \IE\bigl\{\IP
\bigl[Y_i>m\bmid N^*_{n-k,k},N_{n-1}^{(i)}-N^*_{n-k,k}
\bigr]\bmid N_{n-k,k}^* \bigr\}
\\
& =& \IE\bigl\{\IP\bigl[Y_i>m\bmid N^{(i)}_{n-1}
\bigr]\bmid N_{n-k,k}^* \bigr\} \leq2^{-m}.
\end{eqnarray*}\upqed
\end{pf}

\subsubsection*{Uniform plane tree}

It is well known that there are $n!C_{n-1}$ decorated binary trees of
size $2n-1$ as well as labeled plane trees of size $n$ nodes,
where $C_1,C_2,\dots$ are the Catalan numbers. There are various ways
to describe bijections between the two sets. We first give a direct
algorithm to construct a plane tree from a binary tree; see Figure~\ref{fig3}.

%
%
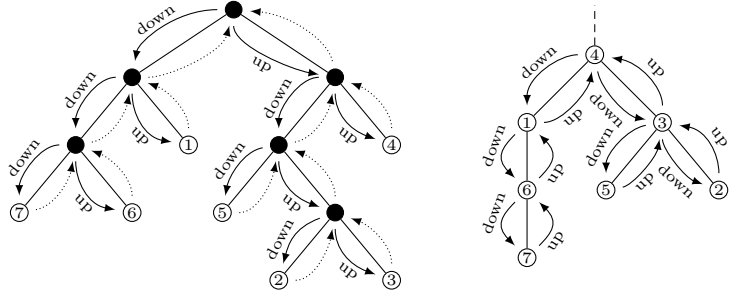
\begin{figure}[b]
\def\up{{\tiny down}}
\def\down{{\tiny up}}
\begin{center}
\begin{tikzpicture}[scale=0.6]\tiny
\tikzstyle{level 1}=[sibling distance=4.5cm]
\tikzstyle{level 2}=[sibling distance=2.5cm]
\node at (0,0) [int] (I1) {\phantom1}
child { node [int] (I5) {\phantom5}
child { node [int] (I6) {\phantom6}
child { node [ext] (E7) { 7 } }
child { node [ext] (E6) { 6 } }
}
child { node [ext] (E1) { 1} }
}
child { node [int] (I3) {\phantom3}
child { node [int] (I4) {\phantom4}
child { node [ext] (E5) { 5} }
child { node [int] (I2) {\phantom2}
child { node [ext] (E2) { 2} }
child { node [ext] (E3) { 3} }
}
}
child { node [ext] (E4) { 4} }
}
;
\draw[-latex,arr] (I1).. controls +(left:13mm) and +(80:8mm).. (I5)
node[above,arrn] {\up};
\draw[-latex,arr] (I5).. controls +(left:8mm) and +(up:8mm).. (I6)
node[above,arrn] {\up};
\draw[-latex,arr] (I6).. controls +(left:8mm) and +(up:8mm).. (E7)
node[above,arrn] {\up};
\draw[-latex,arrd] (E7).. controls +(10:8mm) and +(260:8mm).. (I6);
\draw[-latex,arr] (I6).. controls +(down:8mm) and +(170:10mm).. (E6)
node[below,arrn,pos=0.6] {\down};
\draw[-latex,arrd] (E6).. controls +(up:8mm) and +(340:8mm).. (I6);
\draw[-latex,arrd] (I6).. controls +(10:8mm) and +(260:8mm).. (I5);
\draw[-latex,arr] (I5).. controls +(down:8mm) and +(170:10mm).. (E1)
node[below,arrn,pos=0.6] {\down};
\draw[-latex,arrd] (E1).. controls +(up:8mm) and +(340:8mm).. (I5);
\draw[-latex,arrd] (I5).. controls +(right:8mm) and +(255:10mm).. (I1);
\draw[-latex,arr] (I1).. controls +(down:8mm) and +(170:10mm).. (I3)
node[below,arrn] {\down};
\draw[-latex,arr] (I3).. controls +(left:8mm) and +(up:8mm).. (I4)
node[above,arrn,pos=.6] {\up};
\draw[-latex,arr] (I4).. controls +(190:8mm) and +(80:8mm).. (E5)
node[above,arrn] {\up};
\draw[-latex,arrd] (E5).. controls +(right:8mm) and +(260:8mm).. (I4);
\draw[-latex,arr] (I4).. controls +(down:8mm) and +(160:10mm).. (I2)
node[below,arrn,pos=0.65] {\down};
\draw[-latex,arr] (I2).. controls +(200:10mm) and +(75:8mm).. (E2)
node[above,arrn,pos=0.6] {\up};
\draw[-latex,arrd] (E2).. controls +(right:8mm) and +(260:8mm).. (I2);
\draw[-latex,arr] (I2).. controls +(down:8mm) and +(160:10mm).. (E3)
node[below,arrn,pos=0.65] {\down};
\draw[-latex,arrd] (E3).. controls +(up:8mm) and +(340:8mm).. (I2);
\draw[-latex,arrd] (I2).. controls +(up:8mm) and +(340:8mm).. (I4);
\draw[-latex,arrd] (I4).. controls +(10:8mm) and +(260:8mm).. (I3);
\draw[-latex,arr] (I3).. controls +(down:8mm) and +(160:10mm).. (E4)
node[below,arrn,pos=0.65] {\down};
\draw[-latex,arrd] (E4).. controls +(up:8mm) and +(340:8mm).. (I3);
\draw[-latex,arrd] (I3).. controls +(100:8mm) and +(355:13mm).. (I1);

\tikzstyle{level 1}=[sibling distance=3cm]
\tikzstyle{level 2}=[sibling distance=2.5cm]

\node at (8,0.25) (R) {};
\node at (8,-1) [ext] (C4) {4}
child { node [ext] (C1) {1}
child { node [ext] (C6) {6}
child { node [ext] (C7) {7} }
}
}
child { node [ext] (C5) {3}
child { node [ext] (C2) {5} }
child { node [ext] (C3) {2} }
}
;
\draw[-latex,arr] (C4).. controls +(left:8mm) and +(up:10mm).. (C1)
node[above,arrn] {\up};
\draw[-latex,arr] (C1).. controls +(225:8mm) and +(135:10mm).. (C6)
node[above,arrn,pos=0.4] {\up};
\draw[-latex,arr] (C6).. controls +(225:8mm) and +(135:10mm).. (C7)
node[above,arrn,pos=0.4] {\up};
\draw[-latex,arr] (C7).. controls +(45:8mm) and +(315:10mm).. (C6)
node[below,arrn,pos=0.3] {\down};
\draw[-latex,arr] (C6).. controls +(45:8mm) and +(315:10mm).. (C1)
node[below,arrn,pos=0.3] {\down};
\draw[-latex,arr] (C1).. controls +(right:8mm) and +(260:8mm).. (C4)
node[below,arrn,pos=0.4] {\down};
\draw[-latex,arr] (C4).. controls +(280:8mm) and +(170:8mm).. (C5)
node[below,arrn,pos=0.5] {\up};
\draw[-latex,arr] (C5).. controls +(190:8mm) and +(up:8mm).. (C2)
node[above=1.5pt,arrn,pos=0.65] {\up};
\draw[-latex,arr] (C2).. controls +(10:6mm) and +(260:8mm).. (C5)
node[below,arrn,pos=0.4] {\down};
\draw[-latex,arr] (C5).. controls +(270:6mm) and +(160:8mm).. (C3)
node[below,arrn,pos=0.65] {\up};
\draw[-latex,arr] (C3).. controls +(up:8mm) and +(350:10mm).. (C5)
node[above,arrn,pos=0.5] {\down};
\draw[-latex,arr] (C5).. controls +(up:8mm) and +(350:10mm).. (C4)
node[above,arrn,pos=0.5] {\down};
\draw[densely dashed] (C4) -- (R);
\end{tikzpicture}
\end{center}
\caption{Bijection between a decorated binary tree of size
$2n-1$ (on the left),
and a rooted labeled plane tree of size $n$ (on the right).}
\label{fig3}
\end{figure}

Given a binary tree, we do a depth-first exploration, starting from the
root and exploring left-child before right-child. We construct the
plane tree as we explore the binary tree, starting with an unlabeled
root node. Whenever a left-edge in the binary tree is visited for the
first time, we add one new unlabeled child to the current node in the
plane tree to the right of all existing children of that node, and move
to that new child. If a right-edge is visited for the first time, we
move back to the parent of the current node in the plane tree. Whenever
we encounter a leaf in the binary tree, we copy that label to the node
in the plane tree.

Another way to describe the bijection, initially described between
unlabeled objects, is by means of \emph{Dyck paths} of
length $2(n-1)$. These are syntactically valid strings of $n-1$ nested
bracket pairs. To go from a Dyck path to a binary tree, we parse the
string from left-to-right and at the same time do a depth-first
construction of the binary tree. Start with one active node. Any
opening bracket corresponds to adding a left-child to the currently
active node and then making that child the active node, whereas a
closing bracket corresponds to adding a right-child as sibling of the
left-child that belongs to the opening bracket of the current closing
bracket, and then making that right child the active node. The
labeling is added by inserting $n-1$ of the $n$ leaf labels in front
of the $n-1$ closing brackets, as well as one label at the end of the
string in any of the $n!$ possible orderings. When converting the
labeled Dyck path into a binary tree, every time a label is
encountered that label is copied to the currently active node in the
tree. The Dyck path corresponding to the tree in Figure~\ref{fig3}
would be ``$((()))(()())$'', respectively, with the labeling,
``$(((7)6)1)((5)(2)3)4$''.

To obtain a labeled plane tree from a labeled Dyck path, again do a
depth-first construction, starting with one active node. An opening
bracket corresponds to adding a new child to the currently active node
to the right of all already present siblings and then making that child
the new active node, whereas a closing bracket represents making the
parent of the currently active node the new active node. If a label is
encountered in the string, the label is copied to the currently active node.

%
%
\begin{proposition}\label{prop3} Let $n\geq k\geq1$
and $N^*_{j,k}\sim{\mathcal{P}}^{1}_{j}(0,2k-1)$.
Assume that $X_{n,k}\sim\Bi(N^*_{n-k,k}-(2k-1),1/2 )$. Then
\[
\operatorname{sp}^{k}_{\mathrm{Node}}
\bigl(T^p_n\bigr)~\stackrel{\mathscr{D}} {=}
X_{n,k} + k.
\]
\end{proposition}

\begin{pf} We use the bijection between binary and plane trees. The
number of edges in the spanning tree of $k$ nodes in the plane tree is
equal to the number of left-edges in the spanning tree of the
corresponding $k$ leaves in the binary tree (note that in the spanning
tree of the binary tree, we count left-edges both between internal
nodes as well as between internal nodes and leaves). This is because
only left-edges in the binary tree contribute to the number of edges in
the plane tree. The proof is now a simple consequence of
Proposition~\ref{prop1} and Lemma~\ref{lem1} and the fact that the
number of nodes in any spanning tree is equal to one plus the number of
edges in that spanning tree.
\end{pf}

It is illuminating to see how R\'emy's algorithm acts on plane trees by
means of the bijection described above (see Figure~\ref{fig4}). Apart
from adding new edges to existing nodes, we also observe an operation
that ``cuts'' existing nodes. The trees $T^p_n$ and $T^b_{2n-1}$ are special
cases of Galton--Watson trees (respective offspring distributions
geometric and uniform on $\{0,2\}$) conditioned to have $n$ and $2n-1$
nodes, respectively. As noted by \citet{Janson2006b}, such
conditioned trees cannot in general be grown by only adding edges.
Hence, it is tempting to speculate whether there is a wider class of
offspring distributions for which conditional Galton--Watson trees can
be grown using only local operations on trees such as those in
Figure~\ref{fig4}.

%
%
\begin{figure}[t]
\tikzset{
int/.style = {circle, white, font=\bfseries,
draw=black, fill=black,inner sep=0pt},
ext/.style = {circle, draw, inner sep=0.5pt},
}
\begin{center}
\subfigure[Part of a binary tree]{
\begin{tikzpicture}[scale=0.55]\tiny
\draw[draw=none] (-4,0) -- (4,0);
\draw(-1,-1) -- (-0.5,-0.5);
\draw[densely dashed] (-0.5,-0.5) -- (0,0);
\node at (-1,-1) [int] {\fixbox{a}}
[sibling distance=2cm,level distance=1cm]
child { node (1) at (0.25,0.25) [inner sep=0pt,draw=none] {}}
child { node [int] (2) {\fixbox{b}}
child { node (3) at (0.25,0.25) [inner sep=0pt,draw=none] {}}
child { node [int] (4) {\fixbox{c}}
child { node (5) at (0.25,0.25) [inner sep=0pt,draw=none] {}}
child { node [ext] (6) {\fixbox{z}}
}
}
}
;
\draw(1).. controls +(200:2cm) and +(250:2cm).. (1)
node[above=0.5pt,pos=0.67] {\textit{\footnotesize a}};
\draw(3).. controls +(200:2cm) and +(250:2cm).. (3)
node[above=-0.5pt,pos=0.67] {\textit{\footnotesize b}};
\draw(5).. controls +(200:2cm) and +(250:2cm).. (5)
node[above=0.5pt,pos=0.67] {\textit{\footnotesize c}};
\end{tikzpicture}
}\hfil
\subfigure[Same part of plane tree]{
\begin{tikzpicture}[scale=0.70]\footnotesize
\draw[draw=none] (-4,0) -- (4,0);
\tikzstyle{level 1}=[level distance=1.2cm]
\node at (0,0) [draw=none] {}
child { node [ext] (z) {\fixbox{z}}
[sibling distance=1cm]
child { node (A) [inner sep=0pt,draw=none] {} edge from parent[solid]}
child { node (B) [inner sep=0pt,draw=none] {} edge from parent[solid]}
child { node (C) [inner sep=0pt,draw=none] {} edge from parent[solid]}
edge from parent[dashed]
};
\draw(A).. controls +(245:2cm) and +(295:2cm).. (A)
node[above=1pt,pos=0.5] {\textit{\footnotesize a}};
\draw(B).. controls +(245:2cm) and +(295:2cm).. (B)
node[above=1pt,pos=0.5] {\textit{\footnotesize b}};
\draw(C).. controls +(245:2cm) and +(295:2cm).. (C)
node[above=1pt,pos=0.5] {\textit{\footnotesize c}};
\end{tikzpicture}
}\\
\subfigure[\kern-0.3em Attach at pos. $0$]{
\begin{tikzpicture}[scale=0.50]\footnotesize
\tikzstyle{level 1}=[level distance=1.5cm]
\draw[draw=none] (-2.7,0) -- (2.7,0);
\node at (0,0) [draw=none] {}
child { node [ext] (z) {\fixbox{z}}
[sibling distance=1cm]
child { node at (-0.2,-0.4) [ext,solid] {\fixbox{x}}
edge from parent[solid]}
child { node (A) [inner sep=0pt,draw=none] {} edge from parent[solid] }
child { node (B) [inner sep=0pt,draw=none] {} edge from parent[solid] }
child { node (C) [inner sep=0pt,draw=none] {} edge from parent[solid] }
edge from parent[dashed]
};
\draw(A).. controls +(245:2cm) and +(295:2cm).. (A)
node[above=1pt,pos=0.5] {\textit{\footnotesize a}};
\draw(B).. controls +(245:2cm) and +(295:2cm).. (B)
node[above=1pt,pos=0.5] {\textit{\footnotesize b}};
\draw(C).. controls +(245:2cm) and +(295:2cm).. (C)
node[above=1pt,pos=0.5] {\textit{\footnotesize c}};
\end{tikzpicture}
}
\subfigure[\kern-0.3em Attach at pos. $1$]{
\begin{tikzpicture}[scale=0.5]\footnotesize
\tikzstyle{level 1}=[level distance=1.5cm]
\draw[draw=none] (-2.7,0) -- (2.7,0);
\node at (0,0) [draw=none] {}
child { node [ext] (z) {\fixbox{z}}
[sibling distance=1cm]
child { node (A) [inner sep=0pt,draw=none] {} edge from parent[solid] }
child { node at (0,-0.4) [ext,solid] {\fixbox{x}} edge from parent[solid]
}
child { node (B) [inner sep=0pt,draw=none] {} edge from parent[solid] }
child { node (C) [inner sep=0pt,draw=none] {} edge from parent[solid]}
edge from parent[dashed]
};
\draw(A).. controls +(245:2cm) and +(295:2cm).. (A)
node[above=1pt,pos=0.5] {\textit{\footnotesize a}};
\draw(B).. controls +(245:2cm) and +(295:2cm).. (B)
node[above=1pt,pos=0.5] {\textit{\footnotesize b}};
\draw(C).. controls +(245:2cm) and +(295:2cm).. (C)
node[above=1pt,pos=0.5] {\textit{\footnotesize c}};
\end{tikzpicture}
}
\subfigure[\kern-0.3em Attach at pos. $2$]{
\begin{tikzpicture}[scale=0.5]\footnotesize
\tikzstyle{level 1}=[level distance=1.5cm]
\draw[draw=none] (-2.7,0) -- (2.7,0);
\node at (0,0) [draw=none] {}
child { node [ext] (z) {\fixbox{z}}
[sibling distance=1cm]
child { node (A) [inner sep=0pt,draw=none] {} edge from parent[solid]}
child { node (B) [inner sep=0pt,draw=none] {} edge from parent[solid]}
child { node at (0,-0.4) [ext,solid] {\fixbox{x}} edge from
parent[solid]}
child { node (C) [inner sep=0pt,draw=none] {} edge from parent[solid]}
edge from parent[dashed]
};
\draw(A).. controls +(245:2cm) and +(295:2cm).. (A)
node[above=1pt,pos=0.5] {\textit{\footnotesize a}};
\draw(B).. controls +(245:2cm) and +(295:2cm).. (B)
node[above=1pt,pos=0.5] {\textit{\footnotesize b}};
\draw(C).. controls +(245:2cm) and +(295:2cm).. (C)
node[above=1pt,pos=0.5] {\textit{\footnotesize c}};
\end{tikzpicture}
}
\subfigure[\kern-0.3em Attach at pos. $3$]{
\begin{tikzpicture}[scale=0.5]\footnotesize
\tikzstyle{level 1}=[level distance=1.5cm]
\draw[draw=none] (-2.7,0) -- (2.7,0);
\node at (0,0) [draw=none] {}
child { node [ext] (z) {\fixbox{z}}
[sibling distance=1cm]
child { node (A) [inner sep=0pt,draw=none] {} edge from parent[solid]}
child { node (B) [inner sep=0pt,draw=none] {} edge from parent[solid]}
child { node (C) [inner sep=0pt,draw=none] {} edge from parent[solid]}
child { node at (0.2,-0.4) [ext,solid] {\fixbox{x}} edge from
parent[solid]}
edge from parent[dashed]
};
\draw(A).. controls +(245:2cm) and +(295:2cm).. (A)
node[above=1pt,pos=0.5] {\textit{\footnotesize a}};
\draw(B).. controls +(245:2cm) and +(295:2cm).. (B)
node[above=1pt,pos=0.5] {\textit{\footnotesize b}};
\draw(C).. controls +(245:2cm) and +(295:2cm).. (C)
node[above=1pt,pos=0.5] {\textit{\footnotesize c}};
\end{tikzpicture}
}
\\
\subfigure[Cut at pos. $0$]{
\begin{tikzpicture}[scale=0.5]\footnotesize
\tikzstyle{level 1}=[level distance=1.5cm]
\draw[draw=none] (-2.7,0) -- (2.7,0);
\node at (0,0) [draw=none] {}
child { node [ext] {\fixbox{x}}
child {node at (0,-0.3) [ext,solid] {\fixbox{z}}
[sibling distance=1cm]
child { node (A) [inner sep=0pt,draw=none] {} edge from parent[solid] }
child { node (B) [inner sep=0pt,draw=none] {} edge from parent[solid] }
child { node (C) [inner sep=0pt,draw=none] {} edge from parent[solid] }
edge from parent[solid]
} edge from parent[dashed]};
\draw(A).. controls +(245:2cm) and +(295:2cm).. (A)
node[above=1pt,pos=0.5] {\textit{\footnotesize a}};
\draw(B).. controls +(245:2cm) and +(295:2cm).. (B)
node[above=1pt,pos=0.5] {\textit{\footnotesize b}};
\draw(C).. controls +(245:2cm) and +(295:2cm).. (C)
node[above=1pt,pos=0.5] {\textit{\footnotesize c}};
\end{tikzpicture}
}
\subfigure[Cut at pos. $1$]{
\begin{tikzpicture}[scale=0.5]\footnotesize
\tikzstyle{level 1}=[level distance=1.5cm]
\draw[draw=none] (-2.7,0) -- (2.7,0);
\node at (0,0) [draw=none] {}
child { node [ext] {\fixbox{x}} [sibling distance=1.5cm]
child { node (A) [inner sep=0pt,draw=none] {} edge from parent[solid] }
child {node at (0.2,-0.4)[ext,solid] {\fixbox{z}}
[sibling distance=1cm]
child { node (B) [inner sep=0pt,draw=none] {} edge from parent[solid] }
child { node (C) [inner sep=0pt,draw=none] {} edge from parent[solid] }
edge from parent[solid]
} edge from parent[dashed]};
\draw(A).. controls +(245:2cm) and +(295:2cm).. (A)
node[above=1pt,pos=0.5] {\textit{\footnotesize a}};
\draw(B).. controls +(245:2cm) and +(295:2cm).. (B)
node[above=1pt,pos=0.5] {\textit{\footnotesize b}};
\draw(C).. controls +(245:2cm) and +(295:2cm).. (C)
node[above=1pt,pos=0.5] {\textit{\footnotesize c}};
\end{tikzpicture}
}
\subfigure[Cut at pos. $2$]{
\begin{tikzpicture}[scale=0.5]\footnotesize
\tikzstyle{level 1}=[level distance=1.5cm]
\draw[draw=none] (-2.7,0) -- (2.7,0);
\node at (0,0) [draw=none] {}
child { node [ext] {\fixbox{x}} [sibling distance=1cm]
child { node (A) [inner sep=0pt,draw=none] {} edge from parent[solid] }
child { node (B) [inner sep=0pt,draw=none] {} edge from parent[solid] }
child {node at (0.3,-0.4) [ext,solid] {\fixbox{z}}
[sibling distance=1cm]
child { node (C) [inner sep=0pt,draw=none] {} edge from parent[solid] }
edge from parent[solid]
} edge from parent[dashed]};
\draw(A).. controls +(245:2cm) and +(295:2cm).. (A)
node[above=1pt,pos=0.5] {\textit{\footnotesize a}};
\draw(B).. controls +(245:2cm) and +(295:2cm).. (B)
node[above=1pt,pos=0.5] {\textit{\footnotesize b}};
\draw(C).. controls +(245:2cm) and +(295:2cm).. (C)
node[above=1pt,pos=0.5] {\textit{\footnotesize c}};
\end{tikzpicture}
}
\subfigure[Cut at pos. $3$]{
\begin{tikzpicture}[scale=0.5]\footnotesize
\tikzstyle{level 1}=[level distance=1.5cm]
\draw[draw=none] (-2.7,0) -- (2.7,0);
\node at (0,0) [draw=none] {}
child { node [ext] {\fixbox{x}} [sibling distance=1cm]
child { node (A) [inner sep=0pt,draw=none] {} edge from parent[solid] }
child { node (B) [inner sep=0pt,draw=none] {} edge from parent[solid] }
child { node (C) [inner sep=0pt,draw=none] {} edge from parent[solid] }
child {node at (0.3,-0.4)[ext,solid] {\fixbox{z}}
[sibling distance=1cm]
edge from parent[solid]
} edge from parent[dashed]};
\draw(A).. controls +(245:2cm) and +(295:2cm).. (A)
node[above=1pt,pos=0.5] {\textit{\footnotesize a}};
\draw(B).. controls +(245:2cm) and +(295:2cm).. (B)
node[above=1pt,pos=0.5] {\textit{\footnotesize b}};
\draw(C).. controls +(245:2cm) and +(295:2cm).. (C)
node[above=1pt,pos=0.5] {\textit{\footnotesize c}};
\end{tikzpicture}
}
\end{center}
\caption{R\'emy's algorithm acting on plane trees by means of the
bijection given in Figure~\protect\ref{fig3}. We leave it to the
reader to find the operations in the binary tree as given in \textup{(a)} that
correspond to the operations~\textup{(c)--(j)}.}
\label{fig4}
\end{figure}
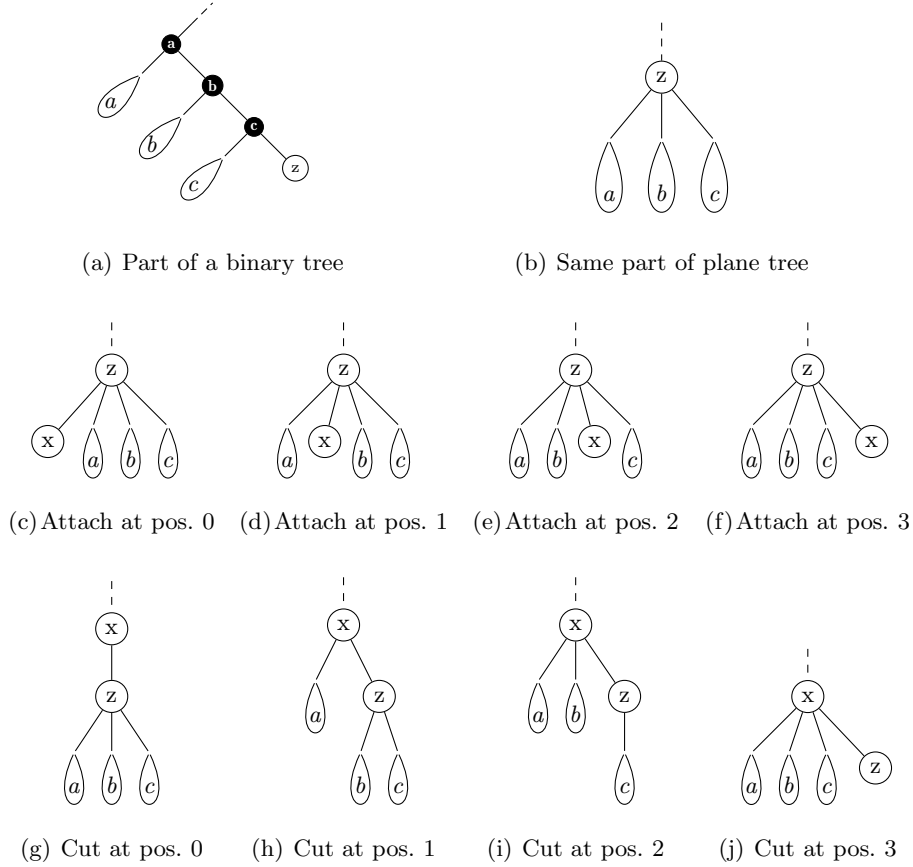

Before proceeding with the proof of Theorem~\ref{thm2}, we need an
auxiliary lemma used
to transfer rates from our urn model to the distributions in
Propositions~\ref{prop2} and~\ref{prop3}.
Here and below $\llVert\cdot\rrVert $ denotes the essential
supremum norm.

%
%
\begin{lemma}\label{lem2} Let $\alpha\geq1$ and $\beta\geq1$.
There is a constant $C=C_{\alpha,\beta}$, such that
for any positive random variable $X$ and any real-valued random
variable $\xi$,
%
%
\begin{eqnarray}\label{15}
&& \dk\bigl(\law(X+\xi),\GG(\alpha,\beta) \bigr)
\nonumber\\[-8pt]\\[-8pt]\nonumber
&&\qquad  \leq C \bigl
(\dk
\bigl(\law(X),\GG(\alpha,\beta) \bigr) + \bigl\llVert\IE\bigl(\xi^2
\mid X\bigr)\bigr\rrVert^{1/2} \bigr).
\end{eqnarray}
If $X$ and $\xi$ satisfy
%
%
\begin{equation}
\label{16} \IP\bigl[\llvert\xi\rrvert\geq t \bmid X \bigr] \leq
c_1 e^{-c_2 t^2 / X}
\end{equation}
for some constants $c_1>0$ and $c_2 > 1$, then
%
%
\begin{eqnarray}\label{17}
&& \dk\bigl(\law(X+\xi),\GG(\alpha,\beta) \bigr)
\nonumber\\[-8pt]\\[-8pt]\nonumber
&&\qquad \leq C \biggl
(\dk
\bigl(\law(X),\GG(\alpha,\beta) \bigr) + \frac{1+c_1+\log
c_2}{\sqrt{c_2}} \biggr).
\end{eqnarray}
\end{lemma}

\begin{pf} The proofs of~\eq{15} and~\eq{17} follow along the lines
of the proof of Lemma~1 of \citet{Bolthausen1982a}.
Once one observes that $\GG(\alpha,\beta)$ has bounded density, the
modifications needed to prove~\eq{15} are straightforward, and hence
omitted. The modifications to prove~\eq{17}, however, are more
substantial, hence we give a complete proof for this case. Let $Z\sim
\GG(\alpha,\beta)$, and let
\begin{eqnarray*}
F(t) &=& \IP[X\leq t], \qquad F^*(t) = \IP[X+\xi\leq t],
\\
G(t) &=& \IP
[Z\leq
t], \qquad\delta= \sup_{t>0} \bigl\llvert F(t)-G(t)\bigr
\rrvert.
\end{eqnarray*}
If $t>\eps>0$, then
\begin{eqnarray*}
F^*(t) &=& \IE\bigl\{\IP[\xi\leq t - X\mid X] \bigr\} \geq\IE\bigl\{\I
[X\leq
t-\eps]\IP[\xi\leq t - X\mid X] \bigr\}
\\
&=& F(t-\eps) - \IE\bigl\{\I[X\leq t-\eps]\IP[\xi>t-X\mid X] \bigr\}.
\end{eqnarray*}
Let $t_0 = \log c_2$ and $\eps= \frac{\log c_2}{\sqrt{c_2}}$,
and observe that, since $c_2>1$, we have $t_0>\eps>0$.
Also note that one can find a constant $c_3$ such that $1-G(t) \leq c_3
e^{-t/2}$.
Using~\eq{16} and setting $M_{\alpha,\beta}$ the maximum
of the density of $\GG(\alpha,\beta)$ (defined explicitly in
Lemma~\ref{lem12} below),
\begin{eqnarray*}
&& \IE\bigl\{\I[X\leq t-\eps]\IP[\xi>t-X\mid X] \bigr\}
\\
&& \qquad\leq\IE\bigl\{\I[X\leq t\wedge t_0-\eps]\IP[\xi>t\wedge
t_0-X|X] \bigr\} + \IP[X>t_0-\eps]
\\
&&\qquad\leq c_1\IE\bigl\{\I[X\leq t\wedge t_0-
\eps]e^{-c_2(t\wedge
t_0-X)^2/X} \bigr\} + \IP[Z>t_0-\eps]+\delta
\\
&&\qquad\leq c_1 e^{-c_2\eps^2/t_0} + \IP[Z>t_0]+\delta+
\eps M_{\alpha,\beta}
\\
&&\qquad\leq c_1 e^{-\log c_2} + \delta+ \frac{c_3+M_{\alpha,\beta
}\log c_2}{\sqrt{c_2}} \leq
\delta+ \frac{c_1+c_3+M_{\alpha,\beta}\log c_2}{\sqrt{c_2}}.
\end{eqnarray*}
Therefore,
\begin{eqnarray*}
F^*(t) - G(t) & \geq & F(t-\eps) - G(t-\eps) - \eps M_{\alpha,\beta} -
\delta-
\frac{c_1+c_3+M_{\alpha,\beta}\log c_2}{\sqrt{c_2}}
\\
& \geq& - 2\delta- \frac{c_1+c_3+2M_{\alpha,\beta}\log c_2}{\sqrt{c_2}}.
\end{eqnarray*}
On the other hand,
\[
F^*(t) \leq F(t+\eps) + \IE\bigl\{\I[ t+\eps< X \leq t_0]\IP[\xi
\leq t-X|X] \bigr\} + \IP[X>t_0].
\]
Since
\begin{eqnarray*}
&& \IE\bigl\{\I[t+\eps< X \leq t_0]\IP[\xi\leq t-X|X] \bigr\}
\\
&&\qquad\leq c_1\IE\bigl\{\I[t+\eps< X \leq t_0]e^{-c_2(t-X)^2/X}
\bigr\} \leq c_1 e^{-c_2\eps^2/t_0} \leq\frac{c_1}{ c_2 }
\end{eqnarray*}
and $\IP[X > t_0] \leq\delta+ c_3 / \sqrt{c_2}$, by a similar
reasoning as above,
\begin{eqnarray*}
F^*(t) - G(t) &\leq& F(t+\eps) + G(t-\eps) + \eps M_{\alpha,\beta} +
\delta+
\frac{c_1+c_3}{\sqrt{ c_2}}
\\
&\leq& 2\delta+ \frac{c_1 + c_3 + M_{\alpha
,\beta}\log c_2}{\sqrt{c_2}}.
\end{eqnarray*}
Hence,
%
%
\begin{equation}
\label{18} \bigl\llvert F^*(t)-G(t)\bigr\rrvert\leq2\delta+
\frac{c_1 + c_3 +
2M_{\alpha,\beta
}\log c_2}{\sqrt{c_2}}.
\end{equation}
From this, one easily obtains~\eq{17}.
\end{pf}

\begin{pf*}{Proof of Theorem~\ref{thm2}}
\textit{Case}~(i). This follows directly from
Proposition~\ref{prop1} and~\eq{1} of Theorem~\ref{thm1}.

\textit{Case} (ii).
Let $W_n=\operatorname{sp}^{k}_{\mathrm
{Node}}(T^b_{2n-1})/\nu_n$ with $\nu_n = \mu
_{n-k-1}(1,2k)$, let $X_{n,k}$ be as in Proposition~\ref{prop2}.
Applying the triangle inequality, we obtain
%
%
\begin{eqnarray}\label{19}
&&\dk\bigl(\law(W_n),\GG(2k,2) \bigr)
\nonumber\\[-8pt]\\[-8pt]\nonumber
&&\qquad\leq\dk\bigl(\law(W_n),\law(X_{n,k}/
\nu_n ) \bigr) + \dk\bigl(\law(X_{n,k}/\nu_n
),\GG(2k,2) \bigr).
\end{eqnarray}
Since the total variation distance is an upper bound on the Kolmogorov
distance, \eq{11} yields that the first term in \eq{19} is of
order $\bigo(n^{-1})$. To bound the second term in~\eq{19},
let $N^*_{n-k,k}$ and $Y_1,\dots,Y_k$ be as in Proposition~\ref
{prop2}; set $X:= N^*_{n-k,k} / \nu_n$ and $\xi:= (Y_1+\cdots+Y_k)
/ \nu_n$. From~\eq{10} and recalling that $ (\sum_{i=1}^k
Y_k )^2\leq k\sum_{i=1}^k Y_i^2$, it is easy to see that
$\IE(\xi^2|X) \leq6k/\nu_n$ almost surely. Applying~\eq{15} from
Lemma~\ref{lem2}, we hence obtain that
\[
\dk\bigl(\law(X_{n,k}/\nu_n),\GG(2k,2)\bigr)\leq C \bigl(
\dk\bigl(\law\bigl(N^*_{n-k,k}/\nu_n\bigr), \GG(2k,2)\bigr)
+ \nu_n^{-1/2} \bigr).
\]
Combining this with Theorem~\ref{thm1} and \eq{19}, the claim follows.

\textit{Case} (iii). Let $N^*_{n-k,k}$ and $X_{n,k}$ be as
in Proposition~\ref{prop3} and let again $\nu_n=\mu_{n-k-1}$. We may\vspace*{1pt}
consider $2X_{n,k}/\mu_n$ in place of $2\operatorname{sp}^{k}_{\mathrm{Node}}(T^p_n)/\nu
_n$, since by \eq{2} of Theorem~\ref{thm1}, the constant
shift $2k/\nu_n$ is of order $n^{-1/2}$, which, by Lemma~\ref{lem2},
translates into an error of order at most $n^{-1/2}$. Let $X:=
N^*_{n-k,k} / \nu_n $ and $\xi:= (2X_n - N^*_{n-k,k})/\nu_n$ and
note that $2X_n / \nu_n = X+\xi$. From Chernoff's inequality, it
follows that~\eq{16} holds with $c_1 = 2$ and $c_2 = \nu_n^2/4$.
For $n$ large enough, $c_2>1$ [again using~\eq{2}] and applying~\eq
{17} from Lemma~\ref{lem2} and~\eq{1} from Theorem~\ref{thm1}, the
claim follows.
\end{pf*}

\section{Random walk: Proof of Theorem~\texorpdfstring{\protect\ref{thm3}}{1.11}} \label{sec2a}

That random walks and random trees are intimately connected has been
observed in many places; see, for example, \citet{Aldous1991} and
\citet{Pitman2006}.
The specific bijections between binary trees and random walk,
excursion, bridge and meander which we will make use of were sketched by
\citet{Marchal2003} and see also the references therein.
It is clear that for each such bijection R\'emy's algorithm can be translated
to recursively create random walk, excursion, bridge and meander of
arbitrary lengths.

%
%
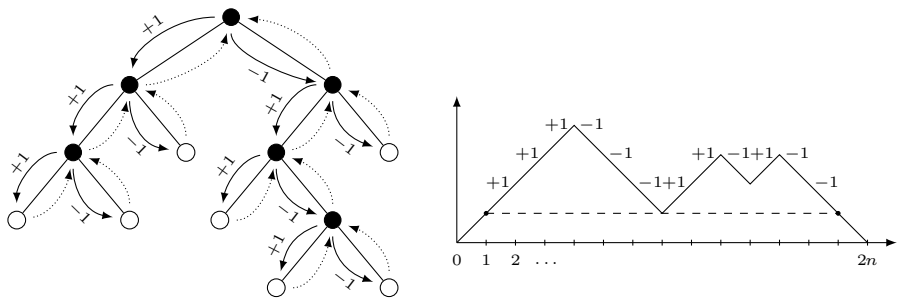
\begin{figure}[b]
\def\up{{\tiny$+1$}}
\def\down{{\tiny$-1$}}
\hfill
\begin{tikzpicture}[scale=0.6]\tiny
\tikzstyle{level 1}=[sibling distance=4.5cm]
\tikzstyle{level 2}=[sibling distance=2.5cm]
\node at (0,0) [int] (I1) {\phantom1}
child { node [int] (I5) {\phantom5}
child { node [int] (I6) {\phantom6}
child { node [ext] (E7) { \phantom7 } }
child { node [ext] (E6) { \phantom6 } }
}
child { node [ext] (E1) {\phantom1} }
}
child { node [int] (I3) {\phantom3}
child { node [int] (I4) {\phantom4}
child { node [ext] (E5) {\phantom5} }
child { node [int] (I2) {\phantom2}
child { node [ext] (E2) {\phantom2} }
child { node [ext] (E3) {\phantom3} }
}
}
child { node [ext] (E4) {\phantom4} }
}
;
\draw[-latex,arr] (I1).. controls +(left:13mm) and +(80:8mm).. (I5)
node[above,arrn] {\up};
\draw[-latex,arr] (I5).. controls +(left:8mm) and +(up:8mm).. (I6)
node[above,arrn] {\up};
\draw[-latex,arr] (I6).. controls +(left:8mm) and +(up:8mm).. (E7)
node[above,arrn] {\up};
\draw[-latex,arrd] (E7).. controls +(10:8mm) and +(260:8mm).. (I6);
\draw[-latex,arr] (I6).. controls +(down:8mm) and +(170:10mm).. (E6)
node[below,arrn,pos=0.6] {\down};
\draw[-latex,arrd] (E6).. controls +(up:8mm) and +(340:8mm).. (I6);
\draw[-latex,arrd] (I6).. controls +(10:8mm) and +(260:8mm).. (I5);
\draw[-latex,arr] (I5).. controls +(down:8mm) and +(170:10mm).. (E1)
node[below,arrn,pos=0.6] {\down};
\draw[-latex,arrd] (E1).. controls +(up:8mm) and +(340:8mm).. (I5);
\draw[-latex,arrd] (I5).. controls +(right:8mm) and +(255:10mm).. (I1);
\draw[-latex,arr] (I1).. controls +(down:8mm) and +(170:10mm).. (I3)
node[below,arrn] {\down};
\draw[-latex,arr] (I3).. controls +(left:8mm) and +(up:8mm).. (I4)
node[above,arrn,pos=.6] {\up};
\draw[-latex,arr] (I4).. controls +(190:8mm) and +(80:8mm).. (E5)
node[above,arrn] {\up};
\draw[-latex,arrd] (E5).. controls +(right:8mm) and +(260:8mm).. (I4);
\draw[-latex,arr] (I4).. controls +(down:8mm) and +(160:10mm).. (I2)
node[below,arrn,pos=0.65] {\down};
\draw[-latex,arr] (I2).. controls +(200:10mm) and +(75:8mm).. (E2)
node[above,arrn,pos=0.6] {\up};
\draw[-latex,arrd] (E2).. controls +(right:8mm) and +(260:8mm).. (I2);
\draw[-latex,arr] (I2).. controls +(down:8mm) and +(160:10mm).. (E3)
node[below,arrn,pos=0.65] {\down};
\draw[-latex,arrd] (E3).. controls +(up:8mm) and +(340:8mm).. (I2);
\draw[-latex,arrd] (I2).. controls +(up:8mm) and +(340:8mm).. (I4);
\draw[-latex,arrd] (I4).. controls +(10:8mm) and +(260:8mm).. (I3);
\draw[-latex,arr] (I3).. controls +(down:8mm) and +(160:10mm).. (E4)
node[below,arrn,pos=0.65] {\down};
\draw[-latex,arrd] (E4).. controls +(up:8mm) and +(340:8mm).. (I3);
\draw[-latex,arrd] (I3).. controls +(100:8mm) and +(355:13mm).. (I1);
\begin{scope}[xshift=5cm,yshift=-5cm,scale=0.65]
\draw[-latex] (0,0) -- (0,5);
\draw[-latex] (0,0) -- (15,0);
\foreach\x in {1,2,...,14}
\draw(\x,-0.1) -- (\x,0.1);
\draw[dashed] (1,1) -- (13,1);
\draw(1,1) --
(2,2) { node [above=2pt, pos=0.4] {$+1$} } --
(3,3) { node [above=2pt, pos=0.4] {$+1$} } --
(4,4) { node [above=2pt, pos=0.4] {$+1$} } --
(5,3) { node [above=2pt, pos=0.6] {$-1$} } --
(6,2) { node [above=2pt, pos=0.6] {$-1$} } --
(7,1) { node [above=2pt, pos=0.6] {$-1$} } --
(8,2) { node [above=2pt, pos=0.4] {$+1$} } --
(9,3) { node [above=2pt, pos=0.4] {$+1$} } --
(10,2) { node [above=2pt, pos=0.6] {$-1$} } --
(11,3) { node [above=2pt, pos=0.4] {$+1$} } --
(12,2) { node [above=2pt, pos=0.6] {$-1$} } --
(13,1) { node [above=2pt, pos=0.6] {$-1$} };
\draw(0,0)--(1,1);
\drawdot{1,1};
\draw(13,1)--(14,0);
\drawdot{13,1};
\node at (0,0) [below=2pt] {$0$};
\node at (1,0) [below=2pt] {$1$};
\node at (2,0) [below=2pt] {$2$};
\node at (3,0) [below=2pt] {$\mathstrut\cdots$};
\node at (14,0) [below=2pt] {$2\kern-0.2ex n$};
\end{scope}
\end{tikzpicture}
\hfill\hfill
\caption{Illustration of the bijection between a binary tree
with $n$ leaves (on the left), and random walk excursions of
length $2n$ (on the right).}
\label{fig5}
\end{figure}

\subsection*{Random walk excursion} The simplest bijection is that
between a binary tree of size $2n-1$ and a
(positive) random walk excursion of length $2n$, as illustrated in
Figure~\ref{fig5}. Note first that the first and last step of the excursion
must be $+1$ and $-1$, respectively, that is, $S^e_{2n}(1) =
S^e_{2n}(2n-1) =
1$. To map the tree to the path from $1$ to $2n-1$, we do a
left-to-right depth-first
exploration of the tree (i.e., counterclockwise): starting from the
root, each time an edge is visited
the first time (out of a total of two times that each edge is visited), the
excursion will go up by one if the edge is a left-edge and go down by
one if
the edge is a right-edge. By means of the Dyck path representation of
the binary tree, we conclude that in this exploration process, the
number of explored left edges (``opening brackets'') is always larger
than the number of explored right edges (``closing brackets''), hence the
random walk stays positive. Furthermore, since the number of left- and
right-edges is equal, the final height is the same as the starting
height. It is not hard to see that the height of a time point in the
excursion corresponds to one plus the number of left-edges from the
corresponding point in the binary tree up to the root.

Furthermore, the pairing between leaves and internal nodes in the
(planted) binary tree induces a pairing between the time points in the
random walk excursion (the pairing in Figure~\ref{fig2}, by means of
the bijection in Figure~\ref{fig5}, results in the pairing in
Figure~\ref{fig6}). Note that all time points can be paired except for
the final time point $2n$ for which, however, we know the height.

%
%
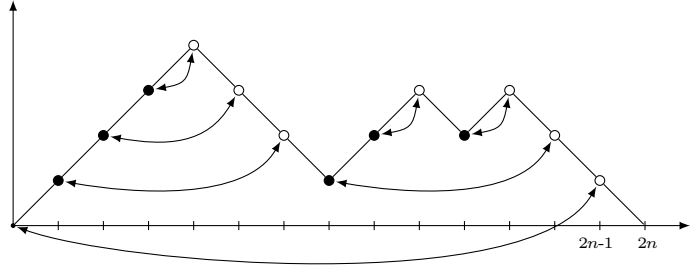
\begin{figure}
\def\up{{\tiny$+1$}}
\def\down{{\tiny$-1$}}
\hfill
\begin{tikzpicture}[xscale=0.6,yscale=0.6]\tiny
\draw[-latex] (0,0) -- (0,5);
\draw[-latex] (0,0) -- (15,0);
\foreach\x in {1,2,...,14}
\draw(\x,-0.1) -- (\x,0.1);
\node at (0,0) [int,inner sep=0.5pt] (RW0) {};
\node at (1,1) [int,inner sep=1.3pt] (RW1) {};
\node at (2,2) [int,inner sep=1.3pt] (RW2) {};
\node at (3,3) [int,inner sep=1.3pt] (RW3) {};
\node at (4,4) [ext,inner sep=1.3pt] (RW4) {};
\node at (5,3) [ext,inner sep=1.3pt] (RW5) {};
\node at (6,2) [ext,inner sep=1.3pt] (RW6) {};
\node at (7,1) [int,inner sep=1.3pt] (RW7) {};
\node at (8,2) [int,inner sep=1.3pt] (RW8) {};
\node at (9,3) [ext,inner sep=1.3pt] (RW9) {};
\node at (10,2) [int,inner sep=1.3pt] (RW10) {};
\node at (11,3) [ext,inner sep=1.3pt] (RW11) {};
\node at (12,2) [ext,inner sep=1.3pt] (RW12) {};
\node at (13,1) [ext,inner sep=1.3pt] (RW13) {};
\node at (14,0) [inner sep=0pt,outer sep=0pt](RW14) {};
\draw(RW0) -- (RW1) -- (RW2) -- (RW3) -- (RW4) -- (RW5) -- (RW6) -- (RW7)
-- (RW8) -- (RW9) -- (RW10) -- (RW11) -- (RW12) -- (RW13) -- (RW14);

\draw[latex-latex, shorten <=0.2mm, shorten >=0.2mm] (RW3)..controls
+(10:8mm) and +(260:8mm)..(RW4);
\draw[latex-latex, shorten <=0.2mm, shorten >=0.2mm] (RW2)..controls
+(-10:14mm) and +(240:14mm)..(RW5);
\draw[latex-latex, shorten <=0.2mm, shorten >=0.2mm] (RW1)..controls
+(-10:17mm) and +(240:17mm)..(RW6);
\draw[latex-latex, shorten <=0.2mm, shorten >=0.2mm] (RW8)..controls
+(10:8mm) and +(260:8mm)..(RW9);
\draw[latex-latex, shorten <=0.2mm, shorten >=0.2mm] (RW10)..controls
+(10:8mm) and +(260:8mm)..(RW11);
\draw[latex-latex, shorten <=0.2mm, shorten >=0.2mm] (RW7)..controls
+(-10:17mm) and +(240:17mm)..(RW12);
\draw[latex-latex, shorten <=0.2mm, shorten >=0.2mm] (RW0)..controls
+(-20:25mm) and +(240:30mm)..(RW13);
\node at (13,0) [below=2pt] {\kern-3pt$2\kern-0.2ex n\kern
-0.1ex\mbox{-}\kern-0.2ex1$};
\node at (14,0) [below=2pt] {\kern2pt$2\kern-0.2ex n$};
\end{tikzpicture}
\hfill\hfill
\caption{Pairing up the points in the random walk excursion. Note that
we pair up time point $2n-1$ with time point $0$, whereas time
point $2n$ is left without a partner.}
\label{fig6}
\end{figure}
%

%
%
\begin{proposition}[(Height of an excursion at a random time)]\label{prop4}
If $n\geq1$, $N^*_{n-1} \sim{\mathcal{P}}^{1}_{n-1}(0,1)$
and $K'\sim\U\{
0,1,\dots,2n-1\}$ independent of $N^*_{n-1}$ and the
excursion $(S^e_{2n}(u))_{u=0}^{2n}$, then
\[
S^e_{2n}\bigl(K'\bigr)\sim\Bi
\bigl(N^*_{n-1},1/2\bigr).
\]
\end{proposition}
\begin{pf} Mapping the pairing of leaves and internal nodes from the
planted binary tree to the excursion, we have that the heights in each
pair differ by exactly one because, by definition of the pairing, each
leaf has one more left edge in its path up to the root as compared to
the internal node it is paired with.

Let $J\sim\Be(1/2)$ independent of all else. Instead of choosing a
random time point $K'$, we may as well choose with probability $1/2$
the time point corresponding to Leaf~1 ($J=0$), and choose with
probability $1/2$ the time point paired with the time point given by
Leaf~1 ($J=1$). Recall that the height of a time point corresponding to
a leaf is just one plus the number of left-edges $M_{1,n}$ in the path
to the root in the corresponding binary tree. From Lemma~\ref{lem1}
with $k=1$, we have $M_{1,n}\sim\Bi(N^*_{n-1}-1,1/2)$. Let $X_n$ be
the height of the excursion at the time point corresponding to the node
chosen in the binary tree; we have $X_n = 1 + M_{1,n} - J$. Since $J$
is independent of the tree and since $1-J\sim\Be(1/2)$, we have $X_n
\sim\Bi(N^*_{n-1},1/2)$, which proves the claim.
\end{pf}

\subsection*{Random walk bridge} We now discuss the bijection between
decorated binary trees and random walk bridges; see
Figure~\ref{fig7} for an example. We first mark the path from Leaf~1
to the root. We call all the internal nodes along this path, including the
root, the \emph{spine} (the trivial tree of size one has no internal
node and,
therefore, an empty spine). As before, a left edge represents ``$+1$''
and a right-edge represents ``$-1$''. The exploration starts at the root.
Whenever a spine node is visited, explore first the child (and its
subtree) that is not part of the spine, and then the child that is next
in the spine. Also, if the right child of a spine node is being
explored and if that child is not itself a spine node do the
exploration \emph{clockwise}, until the exploration process is back to
the spine. This makes each sub-tree to the left of the spine a positive
excursion and each sub-tree to the right a negative excursion; cf.
\citet{Pitman2006}, Exercise 7.4.14.

%
%
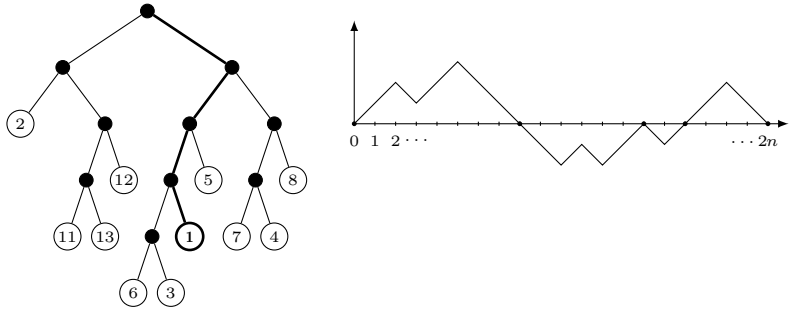
\begin{figure}
\hfill
\begin{tikzpicture}[scale=0.5]\tiny
\tikzstyle{level 1}=[sibling distance=4.5cm]
\tikzstyle{level 2}=[sibling distance=2.25cm]
\tikzstyle{level 3}=[sibling distance=1.0cm]
\node at (0,0) [int] (1) {\phantom{$\cdot$}}
child { node [int] {\phantom{$\cdot$}}
child { node [ext] {\fixbox{2}}}
child { node [int] {\phantom{$\cdot$}}
child { node [int] {\phantom{$\cdot$}}
child { node [ext] {\fixbox{11}}}
child { node [ext] {\fixbox{13}}}
}
child { node [ext] {\fixbox{12}}}
}
}
child { node [int] (2) {\phantom{$\cdot$}}
child { node [int] (3) {\phantom{$\cdot$}}
child { node [int] (4) {\phantom{$\cdot$}}
child { node [int] {\phantom{$\cdot$}}
child { node [ext] {\fixbox{6}}}
child { node [ext] {\fixbox{3}}}
}
child { node [ext,line width=1pt] (5) {\fixbox{\textbf{1}}}}
}
child { node [ext] {\fixbox{5}}}
}
child { node [int] {\phantom{$\cdot$}}
child { node [int] {\phantom{$\cdot$}}
child { node [ext] {\fixbox{7}}}
child { node [ext] {\fixbox{4}}}
}
child { node [ext] {\fixbox{8}}}
}
}
;
\draw[line width=1pt] (1) -- (2) -- (3) -- (4) -- (5);
\begin{scope}[xshift=5.5cm,yshift=-3cm,scale=0.55]
\draw[-latex] (0,0) -- (0,5);
\draw[-latex] (0,0) -- (21,0);
\foreach\x in {1,2,...,20}
\draw(\x,-0.1) -- (\x,0.1);
\draw
(0,0)--
(1,1)--
(2,2)--
(3,1)--
(4,2)--
(5,3)--
(6,2)--
(7,1)--
(8,0)--
(9,-1)--
(10,-2)--
(11,-1)--
(12,-2)--
(13,-1)--
(14,0)--
(15,-1)--
(16,0)--
(17,1)--
(18,2)--
(19,1)--
(20,0);
\drawdot{0,0}; \drawdot{8,0}; \drawdot{14,0}; \drawdot{16,0};
\drawdot{20,0};
\node at (0,0) [below=2pt] {$0$};
\node at (1,0) [below=2pt] {$1$};
\node at (2,0) [below=2pt] {$2$};
\node at (3,0) [below=2pt] {$\cdots$};
\node at (19,0) [below=2pt] {$\cdots$\phantom{1}};
\node at (20,0) [below=2pt] {$2\kern-0.2ex n$};
\end{scope}
\end{tikzpicture}
\hfill\hfill
\caption{Illustration of the bijection between a decorated binary
tree of size $2n+1$ with a spine and random walk bridge of length $2n$.
Note that within sub-trees that grow to the left of the spine, the
depth-first exploration is done counterclockwise, whereas within
sub-trees that grow to the right it is done clockwise.}
\label{fig7}
\end{figure}
%

%
%
\begin{proposition}[(Occupation time of bridge)]\label{prop5}
If $n\geq0$, then
\[
L^b_{2n} \sim{\mathcal{P}}^{1}_n(0,1).
\]
\end{proposition}
\begin{pf} The proof is straightforward by observing that the number of
visits to the origin $L^b_{2n}$ is exactly the number of nodes in the
path from Leaf~1 to the root and then applying Proposition~\ref{prop1}
with $k=1$ and $n$ replaced by $n+1$.
\end{pf}

\subsection*{Random walk meander}

We use a well-known bijection between random walk bridges of
length $2n$ and meanders of length $2n+1$; see Figure~\ref{fig8}.
Start the meander with one positive step. Then, follow the absolute
value of the bridge, except that the last step of every negative
excursion is flipped. Alternatively, consider the random walk bridge
difference sequence. Leave all the steps belonging to positive
excursions untouched, and multiply all steps belonging to negative
excursions by $-1$, except for the last step of each respective
negative excursion (which must necessarily be a ``$+$1''). Now, start the
meander with one positive step and then follow the new difference sequence.

%
%
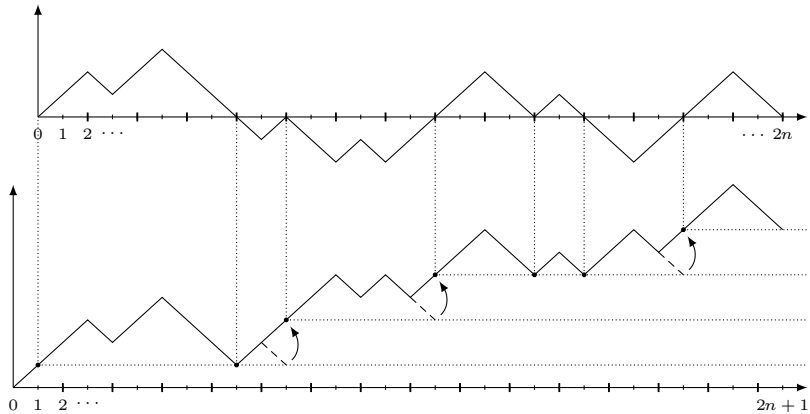
\begin{figure}
\hfill
\begin{tikzpicture}[xscale=.33,yscale=.3]\tiny
\begin{scope}
\draw[-latex] (0,0) -- (0,5);
\draw[-latex] (0,0) -- (31,0);
\foreach\x in {1,2,...,30}
\draw(\x,-0.1) -- (\x,0.1);
\foreach\x in {2,4,...,30}
\draw[line width=0.7pt] (\x,-0.2) -- (\x,0.2);
\draw
(0,0)--(1,1)--(2,2)--(3,1)--(4,2)--(5,3)--(6,2)--
(7,1)--(8,0)--(9,-1)--(10,0)--(11,-1)--(12,-2)--
(13,-1)--(14,-2)--(15,-1)--(16,0)--(17,1)--(18,2)--(19,1)--(20,0)--
(21,1) --
(22,0) --
(23,-1) --
(24,-2) --
(25,-1) --
(26,0) --
(27,1) --
(28,2) --
(29,1) --
(30,0)
;
\node at (0,0) [below=2pt] {$0$};
\node at (1,0) [below=2pt] {$1$};
\node at (2,0) [below=2pt] {$2$};
\node at (3,0) [below=2pt] {$\cdots$};
\node at (29,0) [below=2pt] {$\cdots$\phantom{1}};
\node at (30,0) [below=2pt] {$2\kern-0.2ex n$};
\end{scope}
\begin{scope}[xshift=-1cm,yshift=-12cm]
\draw[-latex] (0,0) -- (0,9);
\draw[-latex] (0,0) -- (32,0);
\foreach\x in {1,2,...,31}
\draw(\x,-0.1) -- (\x,0.1);
\foreach\x in {2,4,...,30}
\draw[line width=0.7pt] (\x,-0.2) -- (\x,0.2);
\draw(0,0) -- (1,1);
\drawdot{1,1};
\begin{scope}[xshift=1cm,yshift=1cm]
\draw[densely dotted] (0,0) -- (31,0) (10,2)--(31,2) (16,4)--(31,4)
(26,6)--(31,6);
\draw[densely dotted] (0,0)--(0,11) (8,0)--(8,11) (10,2)--(10,11)
(16,4)--(16,11) (20,4)--(20,11) (22,4)--(22,11) (26,6)--(26,11);
\drawdot{8,0}; \drawdot{10,2}; \drawdot{16,4}; \drawdot{20,4};
\drawdot{22,4}; \drawdot{26,6};
\draw[densely dashed] (9,1) -- (10,0) (15,3) -- (16,2) (25,5)--(26,4);
\draw[-latex,shorten >=1mm,shorten <=1mm]
(10,0)..controls +(45:8mm) and +(300:10mm)..(10,2);
\draw[-latex,shorten >=1mm,shorten <=1mm]
(16,2)..controls +(45:8mm) and +(300:10mm)..(16,4);
\draw[-latex,shorten >=1mm,shorten <=1mm]
(26,4)..controls +(45:8mm) and +(300:10mm)..(26,6);
\draw
(0,0)--(1,1)--(2,2)--(3,1)--(4,2)--(5,3)--(6,2)--
(7,1)--(8,0)--(9,1)--(10,2)--(11,3)--(12,4)--
(13,3)--(14,4)--(15,3)--(16,4)--(17,5)--(18,6)--(19,5)--(20,4)--
(21,5) --
(22,4) --
(23,5) --
(24,6) --
(25,5) --
(26,6) --
(27,7) --
(28,8) --
(29,7) --
(30,6)
;
\end{scope}
\node at (0,0) [below=2pt] {$0$};
\node at (1,0) [below=2pt] {$1$};
\node at (2,0) [below=2pt] {$2$};
\node at (3,0) [below=2pt] {$\cdots$};
\node at (31,0) [below=2pt] {$2\kern-0.2ex n + 1$};
\end{scope}
\end{tikzpicture}
\hfill\hfill
\caption{Illustration of the bijection between a random walk
bridge of length $2n$ (above) and a meander of length $2n+1$ (below).}
\label{fig8}
\end{figure}
%

%
%
\begin{proposition}[(Final height of meander)]\label{prop6}
If $n\geq0$, $N^*_{n} \sim{\mathcal{P}}^{1}_n(0,1)$, $X_n\sim\Bi
(N^*_{n}-1,1/2)$ and $Y_n\sim\Bi(N^*_{n},1/2)$, then
\[
S^m_{2n+1}(2n+1) \sim\law(2 X_n + 1),\qquad
S^m_{2n+2}(2n+2) \sim\law(2 Y_n\mid
Y_n>0).
\]
\end{proposition}
\begin{pf} It is clear that every negative excursion in the random walk
will increase the final height of the meander by two. Since the number
of negative excursions equals the number of left-edges in the spine of
the corresponding binary tree, the first identity follows directly from
Lemma~\ref{lem1} for $k=1$. To obtain a meander of length $2n+2$,
proceed as follows. First, consider a meander of length $2n+1$,
let $2X_n+1$ be its final height, and then add one additional time step
to the meander by means of an independent fair coin toss. The resulting
process is a simple random walk, conditioned to be positive from
time steps $1$ to $2n+1$. The height of this process at time $2n+2$ has
distribution $2Y_n$, where we can take $Y_n = X_n+J$ and where $J\sim
\Be(1/2)$ independent of $X_n$. However, the final height of this
process may now be zero. Hence, conditioning on the path being positive
results in a meander of length $2n+2$. This proves the second identity.
\end{pf}

%
%
\begin{figure}
\begin{center}
\begin{tikzpicture}[scale=0.55]\tiny
\tikzstyle{level 1}=[sibling distance=4.5cm]
\tikzstyle{level 2}=[sibling distance=2.25cm]
\tikzstyle{level 3}=[sibling distance=1cm]
\node at (0,0) [int] (1) {\phantom{$\cdot$}}
child { node [int] {\phantom{$\cdot$}}
child { node [int] {\phantom{$\cdot$}}
child { node [ext] {\fixbox{7}}}
child { node [ext] {\fixbox{23}}}
}
child { node [int] {\phantom{$\cdot$}}
child { node [int] {\phantom{$\cdot$}}
child { node [ext] {\fixbox{18}}}
child { node [ext] {\fixbox{17}}}
}
child { node [ext] {\fixbox{22}}}
}
}
child { node [int] (2) {\phantom{$\cdot$}}
child { node [int] (3) {\phantom{$\cdot$}}
child { node [int] (4) {\phantom{$\cdot$}}
child { node [int] {\phantom{$\cdot$}}
child { node [ext] {\fixbox{3}}}
child { node [ext] {\fixbox{15}}}
}
child { node [ext,line width=1pt] (5) {\fixbox{\textbf{1}}}}
}
child { node [ext] {\fixbox{8}}}
}
child { node [int] {\phantom{$\cdot$}}
child { node [int] {\phantom{$\cdot$}}
child { node [ext] {\fixbox{9}}}
child { node [ext] {\fixbox{4}}}
}
child { node [ext] {\fixbox{16}}}
}
}
;
\draw[line width=1pt] (1) -- (2) -- (3) -- (4) -- (5);
\tikzstyle{level 1}=[sibling distance=4.5cm]
\tikzstyle{level 2}=[sibling distance=2.25cm]
\tikzstyle{level 3}=[sibling distance=1cm]
\node at (10,0) [int] (M1) {\phantom{$\cdot$}}
child { node [int] (M2) {\phantom{$\cdot$}}
child { node [int] {\phantom{$\cdot$}}
child { node [ext] {\fixbox{19}}}
child { node [ext] {\fixbox{21}}}
}
child { node [int] (M3) {\phantom{$\cdot$}}
child { node [int] (M4) {\phantom{$\cdot$}}
child { node [ext] {\fixbox{2}}}
child { node [ext,line width=1pt] (M5)
{\fixbox{$\boldsymbol{+}$}}}
}
child { node [ext] {\fixbox{14}}}
}
}
child { node [int] {\phantom{$\cdot$}}
child { node [int] {\phantom{$\cdot$}}
child { node [ext] {\fixbox{6}}
}
child { node [ext] {\fixbox{5}}}
}
child { node [int] {\phantom{$\cdot$}}
child { node [int] {\phantom{$\cdot$}}
child { node [ext] {\fixbox{10}}}
child { node [int] {\phantom{$\cdot$}}
child { node [ext] {\fixbox{12}}}
child { node [ext] {\fixbox{20}}}
}
}
child { node [ext] {\fixbox{11}}
}
}
}
;
\draw[line width=1pt] (M1) -- (M2) -- (M3) -- (M4) -- (M5);

\begin{scope}[xshift=-4cm,yshift=-11cm,scale=0.425]
\draw[-latex] (0,0) -- (0,6);
\draw[-latex] (0,0) -- (46,0);
\foreach\x in {1,2,...,45}
\draw(\x,-0.1) -- (\x,0.1);
\foreach\x in {2,4,...,44}
\draw[line width=0.7pt] (\x,-0.2) -- (\x,0.2);
\draw(0,0) --
(1,1) --
(2,2) --
(3,3) --
(4,2) --
(5,1) --
(6,2) --
(7,3) --
(8,2) --
(9,1) --
(10,0) --
(11,-1) --
(12,-2) --
(13,-1) --
(14,-2) --
(15,-1) --
(16,0) --
(17,-1) --
(18,0) --
(19,1) --
(20,2) --
(21,1) --
(22,0);
\draw[line width=1pt] (22,0)--(23,1);
\drawdot{22,0}; \drawdot{23,1};
\draw[densely dashed]
(23,0) --
(24,-1) --
(25,-2) --
(26,-3) --
(27,-2) --
(28,-3) --
(29,-4) --
(30,-3) --
(31,-2) --
(32,-1) --
(33,-2) --
(34,-1) --
(35,0) --
(36,1) --
(37,2) --
(38,1) --
(39,0) --
(40,-1) --
(41,0) --
(42,1) --
(43,0)
;
\draw
(23,1) --
(24,2) --
(25,3) --
(26,4) --
(27,3) --
(28,4) --
(29,5) --
(30,4) --
(31,3) --
(32,2) --
(33,3) --
(34,2) --
(35,3) --
(36,4) --
(37,5) --
(38,4) --
(39,3) --
(40,4) --
(41,5) --
(42,6) --
(43,5)
;
\end{scope}
\end{tikzpicture}
\end{center}
\caption{Bijection between a pair of decorated binary trees with
a total size of $2n+2$ and an unconditional random walk of
length $2n+1$. The meander part of the walk is constructed through a
random walk bridge, which is plotted in dashed lines.}
\label{fig9}
\end{figure}

\subsection*{Unconditional random walk}
In order to represent an unconditional random walk of length $2n+1$, we
use two decorated binary trees, the first tree representing the bridge
part of the random walk (i.e., the random walk until the last return to
the origin) and the second tree representing the meander part (i.e.,
the random walk after the last return to the origin); see Figure~\ref
{fig9}. Note that every random walk of odd length has a meander part.
First, with equal probability, start either with the two trivial trees
\extnode{1} and \extnode{{$\scriptstyle+$}} or with the two trivial
trees \extnode{1} and \extnode{{$\scriptstyle-$}} [representing the
random walk $S_1$ with $S_1(1)=1$, resp., $S_1(1) = -1$]. Then,
perform R\'emys algorithm in exactly the same way as for a single tree.
That is, at each time step, a random node is chosen uniformly among all
nodes of the two trees and then an internal node as well as a new leaf
are inserted. From these two trees, the random walk is constructed in a
straight forward manner: the first tree represents the bridge part,
whereas the second tree represents the meander part (if the initial
second tree was \extnode{{$\scriptstyle-$}}, then the whole meander
is first constructed as illustrated in Figure~\ref{fig8} and is then
flipped to become negative).

%
%
\begin{proposition}[(Occupation time of random walk)]\label{prop7}
If $n\geq0$, then
\[
L_{2n} \sim{\mathcal{P}}^{1}_n(1,1), \qquad
L_{2n+1} \sim{\mathcal{P}}^{1}_n(1,1).
\]
\end{proposition}
\begin{pf} Note that the number of visits to the origin is exactly the
number of nodes in the path from Leaf $1$ (which is always in the first
tree) to the root. Hence, we can use a similar urn embedding as for
Proposition~\ref{prop1} with $k=1$, except that at the beginning the
urn contains one black ball and one white ball (the black ball
representing the leaf of the second tree).

This proves the second identity of the proposition. To obtain the first
identity, take a random walk of length $2n+1$ and remove the last time
step, obtaining a random walk of length $2n$. Since the number of
visits to the origin cannot be changed in this way, the first identity follows.
\end{pf}

%
%
\begin{remark}
Proposition~\ref{prop5} is implicitly used in
\citet{Pitman2006}, Exercise~7.4.14. The other
propositions do
not appear to have been stated explicitly in the literature.
\end{remark}

\begin{pf*}{Proof of Theorem~\ref{thm3}}
Cases (i) and (ii) are immediate from Theorem~\ref{thm1} in
combination with Proposition~\ref{prop7} and Proposition~\ref{prop5},
respectively. Using Proposition~\ref{prop4}, case (iii) is proved in
essentially the same way as case (iii) of Theorem~\ref{thm2}, also
noting that the total variation error introduced by using $K$ instead
of $K'$ is of order $\bigo(n^{-1})$. Using Proposition~\ref{prop6},
case (iv) for odd $n$ is also proved in essentially the same way as
case (iii) of Theorem~\ref{thm2}.

In order to prove case (iv) for even $n$, note that the total
variation distance between $\law(Y_n)$ and $\law(Y_n|Y_n>0)$ is $\IP
[Y_n=0]=\IE2^{-N_n^*}$. Let $Z\sim\GG(2,2)$; using Theorem~\ref{thm1},
\begin{eqnarray*}
\IE2^{-N_n^*} & \leq&\IP\bigl[N_n < \tfrac{1}{2}
\log_2 n\bigr] + 2^{-{1/2}\log_2
n}
\\
& \leq&\IP\bigl[Z < \tfrac{1}{2}\mu^{-1}_n
\log_2 n\bigr] + \dk\bigl(\law(N_n/\mu
_n),\law(Z) \bigr) + n^{-1/2} =\bigo\bigl(n^{-1/2}
\bigr).
\end{eqnarray*}
Now, estimating $\dk(\law(2 Y_n), \GG(2,2) )$ again
follows the
proof of case (iii) of Theorem~\ref{thm2}.
\end{pf*}

\section{Proof of urn Theorem~\texorpdfstring{\protect\ref{thm1}}{1.2}}\label{sec2}

In order to prove Theorem~\ref{thm1}, we need a few lemmas.

%
%
\begin{lemma}\label{lem3}
Let $b\geq0$, $w>0$, $\urn_n=\urn_{n}(b,w)\sim{\mathcal
{P}}^l_n(b,w)$ and
let $n_{i}=n_i(b,w)=w+b+i+\lfloor i/l \rfloor$ be the total
number of balls in the $ {\mathcal{P}}^l_n(b,w)$ urn after the $i$th draw.
If $m\geq1$ is an integer and $D_{n,m}(b,w):=\prod_{i=0}^{m-1}
(i+\urn_{n}(b,w))$, then
%
%
\begin{equation}
\label{20} \IE D_{n,m}(b,w) = \prod_{j=0}^{m-1}
(w+j) \prod_{i=0}^{n-1}\bigl(1+m/n_{i}(b,w)
\bigr)
\end{equation}
and for some positive values $c:=c(b,w,l,m)$ and $C:=C(b,w,l,m)$ not
depending on $n$ we have
%
%
\begin{equation}
\label{21} cn^{ml/(l+1)}<\IE\bigl[\urn_{n}(b,w)^m
\bigr]<Cn^{ml/(l+1)}.
\end{equation}
\end{lemma}

\begin{pf} Fix $b,w$ and write $D_{n,m}=D_{n,m}(b,w)$. We first
prove~\eq{20}. Conditioning on the contents of the urn after draw and
replacement $n-1$, and noting that at each step, the number of white
balls in the urn either stay the same
or increase by exactly one, we have
\begin{eqnarray*}
\IE\{D_{n,m}|\urn_{n-1}\} &=& \frac{\urn_{n-1}}{n_{n-1}}
\frac{ D_{n-1,m}(\urn_{n-1}+m)}{\urn_{n-1}}+ \frac{n_{n-1}-\urn
_{n-1}}{n_{n-1}}D_{n-1,m}
\\
&=&(1+m/n_{n-1}) D_{n-1,m},
\end{eqnarray*}
which when iterated yields~\eq{20}.

By the definition of $n_i$,
\[
i+w+b-1+i/l\leq n_{i}\leq i+w+b+i/l,
\]
and now setting $x=l/(l+1)$ and $y=(w+b-1)l/(l+1)$, we find for some
constants $c, C$ not depending on $n$ that
%
%
\begin{equation}
\qquad c n^{mx}\leq c \frac{\Gamma(mx+y+x+n)}{\Gamma(y+x+n)} \leq\IE D_{n,m}
\leq C
\frac{ \Gamma(mx+y+n)}{\Gamma(y+n)} \leq C n^{mx}.\label{22}
\end{equation}
The upper bound follows from this and the easy fact that $\IE\urn
_n^m\leq\IE D_{n,m}$.
The lower bound follows from~\eq{22} and the following inequality
which follows from Jensen's inequality $\IE\urn_n^m = \IE
D_{n,1}^m\geq(\IE D_{n,1} )^m$.
\end{pf}

Our next result implies that biasing the distribution ${\mathcal{P}}^l_n(b,w)$
against the $r$ rising factorial
is the same as adding $r$ white balls to the urn before starting the
process, and then removing $r$ white balls at the end. We will only use
the lemma for $r=l+1$, but state and prove it for general $r$ because
it is an interesting result in its own right.

%
%
\begin{lemma}\label{lem4}
Let $\urn_{n}(b,w)$ and $D_{n,m}(b,w)$ be as in Lemma~\ref{lem3} and
let $r\geq2$.
If $\urn_n^{[r]}=\urn_n^{[r]}(b,w)$ is a random variable such that
%
%
\begin{equation}
\label{23} \IP\bigl[\urn_n^{[r]}=k\bigr]=
\frac{ [\prod_{i=0}^{r-1}(k+i) ]\IP
[\urn_n(b,w)=k]}{\IE D_{n,r}(b,w)},
\end{equation}
then
%
%
\begin{equation}
\urn_n(b,w+r)\ed\urn_n^{[r]}(b,w)+r.
\label{24}
\end{equation}
\end{lemma}

\begin{pf}
Since $\urn_n(b,w+r)$ and $\urn_n^{[r]}(b,w)+r$ are bounded variables,
the lemma follows by verifying their factorial moments are equal.
With $n_i(b,w)$ as in Lemma~\ref{lem3}, for any $m\geq1$ we have
\begin{eqnarray*}
\IE\prod_{i=0}^{m-1} \bigl(
\urn_n^{[r]}(b,w)+r+i\bigr)&=&\frac{\IE
D_{n,m+r}(b,w)}{\IE D_{n,r}(b,w)}
\\
&=& \prod
_{j=0}^{m-1} (w+r+j) \prod
_{i=1}^{n}\frac
{n_{i-1}(b,w)+m+r}{n_{i-1}(b,w)+r}
\\
&=&\IE D_{n,m}(b,w+r) = \IE\prod_{i=0}^{m-1}
\bigl(i+\urn_n(b,w+r)\bigr);
\end{eqnarray*}
the second and third equalities follow by~\eq{20} and the definition
of $n_i(b,w)$, and the last follows
from the definition of $D_{n,m}(b,w)$.
\end{pf}

%
%
\begin{lemma}\label{lem5}
For $\urn_n(1,w)\sim{\mathcal{P}}^l_n(1,w)$ and $l\geq1$,
there is a coupling of $\urn_n^{(l+1)}(1,w)$, a random variable having
the $(l+1)$-power bias distribution of $\urn_n(1,w)$, with a
variable $\urn_{n-l}(1,w+l+1) \sim{\mathcal{P}}^l_{n-l}(1,w+l+1)$
such that
for some constant $C:=C(w,l)$,
\[
\IP\bigl[\bigl\llvert\urn_{n-l}(1,w+l+1)- \urn_n^{(l+1)}(1,w)
\bigr\rrvert>2l+1 \bigr]\leq Cn^{-l/(l+1)}.
\]
\end{lemma}

\begin{pf}
Obviously, we can couple $\urn_n(1,w+l+1)\sim{\mathcal{P}}^l_n(1,w+l+1)$
with $\urn_{n-l}(1,w+l+1)$ so that
\[
\bigl\llvert\urn_{n-l}(1,w+l+1)-\urn_n(1,w+l+1)\bigr
\rrvert\leq l,
\]
and then Lemma~\ref{lem4}
implies that we may couple $\urn_{n}(1,w+l+1)$ with $\urn
_n^{[l+1]}(1,w)$ [with distribution defined at~\eq{23}] so that almost surely
\begin{eqnarray*}
&&\bigl\llvert\urn_{n-l}(1,w+l+1)-\urn_n^{[l+1]}(1,w)
\bigr\rrvert
\\
&& \qquad\leq\bigl\llvert\urn_{n-l}(1,w+l+1)-\bigl(\urn
_n^{[l+1]}(1,w)+l+1\bigr)\bigr\rrvert+l+1
\\
&&\qquad=\bigl\llvert\urn_{n-l}(1,w+l+1)-\urn_n(1,w+l+1)
\bigr\rrvert+l+1\leq2l+1.
\end{eqnarray*}
And we show
%
%
\begin{equation}
\dtv\bigl(\law\bigl(\urn_n^{[l+1]}(1,w) \bigr),\law\bigl(
\urn_n^{(l+1)}(1,w) \bigr) \bigr)\leq C n^{-l/(l+1)},
\label{25}
\end{equation}
where $\dtv$ is the total variation distance, which for integer-valued
variables $X$ and $Y$ can be defined in two ways:
\[
\dtv\bigl(\law(X),\law(Y)\bigr)=\frac{1}2\sum
_{z\in\IZ} \bigl\llvert\IP[X=z]-\IP[Y=z]\bigr\rrvert=\inf
_{(X,Y)}\IP[X\neq Y];
\]
here, the infimum is taken over all possible couplings of $X$ and $Y$.
Due to the
latter definition,~\eq{25} will imply the lemma since
\begin{eqnarray*}
&&\IP\bigl[\bigl\llvert\urn_{n-l}(1,w+l+1)-\urn_n^{(l+1)}(1,w)
\bigr\rrvert>2l+1 \bigr]
\\
&&\qquad=\IP\bigl[\bigl\llvert\urn_{n-l}(1,w+l+1)-\urn
_n^{(l+1)}(1,w)\bigr\rrvert
\\
&&\qquad >2l+1, \urn_n^{[l+1]}(1,w)
\neq\urn_n^{(l+1)}(1,w) \bigr]
\\
&&\qquad\leq\IP\bigl[\urn_n^{[l+1]}(1,w)\neq\urn
_n^{(l+1)}(1,w) \bigr].
\end{eqnarray*}

Let $\nu_m=\IE\urn_n^m(1,w)$ and note that we can write $\prod
_{i=0}^{l}(x+i)=\sum_{i=0}^{l+1} a_i x^i$ for nonnegative coefficients $a_i$
with $a_{l+1}=1$ (these coefficients are the unsigned Stirling
numbers). Also note that
for nonnegative integers $k$ and $0\leq i\leq l+1$, we have $k^i\leq
k^{l+1}$, and hence $\nu_i\leq\nu_{l+1}$.
Thus,
\begin{eqnarray*}
&&2\dtv\bigl(\law\bigl(\urn_n^{[l+1]}(1,w) \bigr),\law
\bigl(\urn_n^{(l+1)}(1,w) \bigr) \bigr)
\\
&&\qquad= \sum_{k\geq0} \bigl\llvert\IP\bigl[
\urn_n^{[l+1]}(1,w)=k\bigr]-\IP[\urn_n^{(l+1)}(1,w)=k)
\bigr\rrvert
\\
&&\qquad=\sum_{k} \biggl\llvert\frac{\prod_{i=0}^{l}(k+i)}{\IE
D_{n,l+1}(1,w)}
- \frac{k^{l+1}}{\nu_{l+1}}\biggr\rrvert\IP\bigl[\urn_n(1,w)=k\bigr]
\\
&&\qquad=\sum_{k} \Biggl\llvert
\Biggl(k^{l+1}+\sum_{i=0}^{l}
a_i k^i\Biggr)\nu_{l+1}-k^{l+1}
\Biggl(\nu_{l+1}+\sum_{i=0}^{l}
a_i \nu_i\Biggr)\Biggr\rrvert\frac{\IP
[\urn_n(1,w)=k]}{\nu_{l+1}\IE D_{n,l+1}(1,w)}
\\
&& \qquad\leq C\nu_{l}/\IE D_{n,l+1}(1,w) \leq
Cn^{-l/(1+l)},
\end{eqnarray*}
where the last line follows from~\eq{21} of Lemma~\ref{lem3}. This
proves the lemma.
\end{pf}

Below let ${\mathcal{P}}_n(b,w)$ be the distribution of the number of
white balls
in the classical P\'olya urn started with $b$ black balls and $w$ white
balls after $n$ draws. Recall that in the classical P\'olya urn balls
are drawn and returned to the urn along with an additional ball of the
same color
[the notation is to suggest ${\mathcal{P}}^\infty_n(b,w)={\mathcal
{P}}_n(b,w)$].

%
%
\begin{lemma}\label{lem6}
There is a coupling $(Q_{w}(n), n V_w)_{n\geq1}$ with $Q_{w}(n) \sim
{\mathcal{P}}_n(1,w)$ and $V_w\sim\B(w,1)$ such that $|Q_{w}(n)-n
V_w|\leq w+1$
for all $n$ almost surely.
\end{lemma}

\begin{pf}
Using \citet{Feller1968}, equation~(2.4), page~121,
for $w\leq t
\leq w+n$ we obtain
%
%
\begin{equation}
\label{26} \IP\bigl[Q_{w}(n)\leq t\bigr]=\prod
_{i=0}^{w-1} \frac{t-i}{n+w-i}.
\end{equation}
For $U_0, U_1, \ldots, U_{w-1}$ i.i.d. uniform $(0,1)$ variables, we
may set
\[
Q_w(n)=\max_{i=0,1,\ldots, w-1} \bigl(i+\bigl
\lceil(n+w-i)U_i\bigr\rceil\bigr),
\]
since it is not difficult to verify that this gives the same cumulative
distribution function as in~\eq{26}.
By a well-known
representation of the beta distribution, we can take $V_w=\max
(U_{0},\ldots, U_{w-1})$,
and with this coupling the claim follows.
\end{pf}

%
%
\begin{lemma}\label{lem7}
If\/ $\urn_{n}(0,w+1) \sim{\mathcal{P}}^l_{n}(0,w+1)$
then
\[
{\mathcal{P}}_n^l(1,w) = {\mathcal{P}}
_{\urn_{n}(0,w+1)-w-1}(1,w).
\]
\end{lemma}

\begin{pf}
Consider an urn with $1$ black ball and $w$ white balls. Balls are
drawn from the urn and replaced as follows. After the $m$th
ball is drawn, it is replaced in the urn along with another ball of
the same color plus, if $m$ is divisible by $l$, an additional green
ball. If $H$ is the number of times a nongreen ball is drawn in $n$
draws, the number of white balls
in the urn after $n$ draws is distributed as ${\mathcal{P}}_H(1,w)$.
The lemma follows after noting $H+w+1$ is distributed as ${\mathcal{P}}
^l_{n}(0,w+1)$ [which by definition is the distribution of $\urn
_{n}(0,w+1)$] and the number of white
balls in the urn after $n$ draws has distribution ${\mathcal{P}}^l_{n}(1,w)$.
\end{pf}

\begin{pf*}{Proof of Theorem~\ref{thm1}}
The asymptotic $\IE\urn_n^k \asymp n^{kl/(l+1)}$ is~\eq{21} of
Lemma~\ref{lem3}.
We now show that
\[
\lim_{n\to\infty}\frac{\IE\urn_n^{l+1}}{n^l} = w \biggl(\frac
{l+1}{l}
\biggr)^{l}.
\]
The asymptotic $\IE\urn_n^k \asymp n^{kl/(l+1)}$
implies that
\[
\frac{\IE\urn_n^{l+1}}{n^l}=\frac{\IE\prod_{i=0}^{l} (i+\urn
_n)}{n^{l}} + \lito(1).
\]
The numerator in the fraction on the right-hand side of the equality
can be
written using~\eq{20} from Lemma~\ref{lem3} with $b=1, w=w$ and $m=l+1$
as
\begin{eqnarray*}
\IE\prod_{i=0}^{l} (i+
\urn_n) &=& \frac{\Gamma(w+l+1)}{\Gamma(w)}\prod_{i=0}^{n-1+\lfloor{ \vfrac
{n-1}{l}}\rfloor}
\frac{ w+1+i+l+1}{w+1+i}
\\
&&{}\times \prod_{k=1}^{\lfloor
{ \vfrac{n-1}{l}}\rfloor}
\frac{w+1+k l + k-1}{w+1+k l + k+l},
\end{eqnarray*}
and simplifying, especially noting the telescoping product in the final
part of the term (which critically depends on having taken $m=l+1$), we have
\begin{eqnarray*}
\IE\prod_{i=0}^{l} (i+
\urn_n) &=& \frac{\Gamma(w+l+1)}{\Gamma(w)} \frac{\Gamma(w+2+l+n+\lfloor
{ \vfrac{n-1}{l}}\rfloor)\Gamma(w+1)}{\Gamma(w+l+2)\Gamma
(w+1+n+\lfloor{ \vfrac{n-1}{l}}\rfloor)}
\\
&&{}\times  \frac{w+1+l}{w+l+1+\lfloor{
\vfrac{n-1}{l}}\rfloor(l+1)}
\\
&=& w \frac{\Gamma(w+1+l+n+\lfloor{ \vfrac{n-1}{l}}\rfloor)}{\Gamma
(w+1+n+\lfloor{ \vfrac{n-1}{l}}\rfloor)}
\\
&&{}\times \frac{w+1+l+n+\lfloor{
\vfrac{n-1}{l}}\rfloor}{w+l+1+\lfloor{ \vfrac{n-1}{l}}\rfloor(l+1)}.
\end{eqnarray*}
The asymptotic for $\IE\urn_n^{l+1}$ now follows by taking the limit
as $n\to\infty$, using the well-known fact that, for $a>0$, $\lim_{x\to
\infty} \frac{\Gamma(x+a)}{\Gamma(x) x^a}=1$
with $x=w+1+n+\lfloor{ \frac{n-1}{l}}\rfloor$.

The claimed asymptotic for $\mu_n$ follows directly from that of $\IE
\urn_n^{l+1}$, and with
the order of the scaling $\mu_n$ in hand, the lower bound of
Theorem~\ref{thm1} follows from \citet{Pekoz2013}, Lemma~4.1,
which says that for a sequence of scaled integer valued random
variables $(a_n \urn_n)$,
if $a_n\to0$ and $\nu$ is a distribution with density bounded away
from zero on some interval, then there
is a positive constant $c$ such that $\dk(\law(a_n \urn_n), \nu
)\geq c a_n$.

To prove the upper bound we will invoke Theorem~\ref{thm4} and so we
want to closely couple variables having marginal distributions equal
to those of $\urn_n/\mu_n$ and $\urn^*=V_w \urn_n^{(l+1)}/\mu_n$.
Lemma~\ref{lem6} implies there is a coupling of
variables $(Q_w(n))_{n\geq1}$
with corresponding marginal distributions $({\mathcal
{P}}_n(1,w))_{n\geq1}$ satisfying
\[
\bigl\llvert V_w \urn_n^{(l+1)}-
Q_w\bigl(\urn_n^{(l+1)}\bigr)\bigr\rrvert\leq
w+1\qquad\mbox{almost surely.}
\]
Further, by Lemma~\ref{lem5} we can construct a variable $\urn
_{n-l}(1,w+l+1) \sim{\mathcal{P}}^l_{n-l}(1,w+l+1)$ such that
\[
\IP\bigl[\bigl\llvert Q_w\bigl(\urn_{n-l}(1,w+l+1)
\bigr)- Q_w\bigl(\urn_n^{(l+1)}\bigr)\bigr
\rrvert>2l+1 \bigr] \leq Cn^{-l/(l+1)};
\]
here we used that $\llvert Q_w(s)-Q_w(t)\rrvert \leq\llvert
s-t\rrvert $.
Recalling that $P^l_{n-l}(1,w+l+1) = P^l_{n}(0,w+1)$, Lemma~\ref{lem7}
says that we can set $\urn_n=Q_w(\urn_{n-l}(1,w+l+1)-w-1)$ and it is
immediate that
%
\begin{eqnarray}
&&\bigl\llvert Q_w\bigl(\urn_{n-l}(1,w+l+1)-w-1
\bigr)-Q_w\bigl(\urn_{n-l}(1,w+l+1)\bigr)\bigr\rrvert
\leq w+1\nonumber
\\
\eqntext{\mbox{almost surely.}}
\end{eqnarray}
Thus, if we set $b=(2w+2l+3)/\mu_n$ then using the couplings above we find
\[
\IP\bigl[ \bigl\llvert{\urn_n}/{\mu_n}-{V_w
\urn_n^{(l+1)}}/{\mu_n}\bigr\rrvert> b \bigr]
\leq C \mu_n^{-1} \leq Cn^{-l/(l+1)},
\]
where the last inequality follows from~\eq{21} of Lemma~\ref{lem3}
which also implies $b \leq Cn^{-l/(l+1)}$.
Using these couplings and the value of $b$ in Theorem~\ref{thm4}
completes the proof.
\end{pf*}

\section{Stein's method and proof of Theorem~\texorpdfstring{\protect\ref{thm4}}{1.16}}\label{sec3}

We first provide a general framework to develop Stein's method
for $\log$-concave densities. The generalized gamma is a special case
of this class. We use the density approach which is due to Charles
Stein [see \citet{Reinert2005}]. This approach has already been
discussed in other places in greater generality; see, for example,
\citet{Chatterjee2011a}, \citet{Chen2011} and \citet
{Dobler2012}. However, it seems to have gone unnoticed, at least
explicitly, that the approach can be developed much more directly
for $\log$-concave densities.

\subsection{Density approach for $\log$-concave distributions}\label{sec4}

Let $B$ be a function on the interval $(a,b)$ where $-\infty\leq a < b
\leq
\infty$. Assume also $B$ is absolutely continuous on $(a,b)$, $C_B =
\int_a^b e^{-B(z)}\,dz <\infty$ and
$B(a):=\lim_{x\to a^+} B(x)$ and $B(b):=\lim_{x\to b^-} B(x)$ exist
as values in $\IR\cup\{\infty\}$ and
we use these to extend the domain of $B$ to $[a,b]$. Assume that $B$
has a left-continuous derivative on $(a,b)$, denoted by $B'$. From $B$,
we can construct a distribution $P_B$ with probability density
function
\[
\varphi_B(x) = C_B e^{-B(x)}, \qquad a<x<b
\qquad\mbox{where }C_B^{-1} = \int_{a}^{b}
e^{-B(z)}\,dz.
\]
Let $L^1(P_B)$ be the set of measurable functions $h$ on $(a,b)$ such
that
\[
\int_a^b
\bigl\llvert h(x)\bigr\rrvert e^{-B(x)}\,dx < \infty.\]
The
distribution $P_B$ is $\log
$-concave if and only if $B$ is convex. However, before dealing with
this special case, we state a few more general results.

%
%
\begin{proposition}\label{prop8}
If $Z\sim P_B$, we have
\[
\IE\bigl\{f'(Z) - B'(Z)f(Z) \bigr\} = 0
\]
for all functions $f$ for which the expectations exists and for which
\[
\lim_{x\to a^+}f(x)e^{-B(x)}=\lim_{x\to b^-}f(x)e^{-B(x)}=0.
\]
\end{proposition}

\begin{pf}Integration by parts. We omit the straightforward details.
\end{pf}

Now, for $h\in L^1(P_B)$ and $Z\sim P_B$, let
\[
\tilde h(x) = h(x)-\IE h(Z)
\]
and, for $x\in(a,b)$,
%
%
\begin{equation}
\label{27} f_h(x) = e^{B(x)} \int_a^x
\tilde h(z)e^{-B(z)}\,dz = - e^{B(x)} \int_x^b
\tilde h(z)e^{-B(z)}\,dz.
\end{equation}
The key fact is that $f_h$ satisfies the differential (Stein) equation
%
%
\begin{equation}
\label{28} f_h'(x) - B'(x)f_h(x)
= \tilde h(x),\qquad x\in(a,b).
\end{equation}
Define the Mills's-type ratios
%
%
\begin{equation}
\label{29} \kappa_a(x) = e^{B(x)}\int
_a^x e^{-B(z)}\,dz, \qquad
\kappa_b(x) = e^{B(x)}\int_x^b
e^{-B(z)}\,dz.
\end{equation}
From~\eq{27} and~\eq{28}, we can easily deduce the following
nonuniform bounds.

%
%
\begin{lemma}\label{lem8} If $h\in L^1(P_B)$ is bounded, then for
all $x\in(a,b)$,
%
%
\begin{eqnarray}
\bigl\llvert f_h(x)\bigr\rrvert&\leq&\llVert\tilde h\rrVert
\bigl(\kappa_a(x)\wedge\kappa_b(x) \bigr),
\label{30}
\\
\bigl\llvert f'_h(x)\bigr\rrvert&\leq&\llVert\tilde
h\rrVert\bigl\{1+ \bigl\llvert B'(x)\bigr\rrvert\bigl(
\kappa_a(x)\wedge\kappa_b(x) \bigr) \bigr\}.
\label{31}
\end{eqnarray}
\end{lemma}

In the case of convex functions, we can easily adapt the proof of
\citet{Stein1986} to obtain the following uniform bounds.

%
%
\begin{lemma}\label{lem9} If $B$ is convex on $(a,b)$ with unique
minimum $x_0\in[a,b]$, then for any $h\in L^1(P_B)$,
%
%
\begin{equation}
\llVert f_h\rrVert\leq\llVert\tilde h\rrVert\frac{e^{B(x_0)}}{C_B},
\qquad\bigl\llVert B'f_h\bigr\rrVert\leq\llVert
\tilde h\rrVert, \qquad\bigl\llVert f'_h\bigr\rrVert
\leq2\llVert\tilde h\rrVert. \label{32}
\end{equation}
\end{lemma}

\begin{pf} By convexity, we clearly have
%
%
\begin{equation}
\label{33} x_0\leq x \leq z \leq b\quad\Longrightarrow\quad B(x)
\leq B(z)\quad\mbox{and}\quad B'(x)\leq B'(z).
\end{equation}
This implies that for $x>x_0$
\[
\int_x^b e^{-B(z)} \,dz \leq\int
_x^b\frac{B'(z)}{B'(x)}e^{-B(z)}\,dz =
\frac{e^{-B(x)}-e^{-B(b)}}{B'(x)} \leq\frac{e^{-B(x)}}{B'(x)},
\]
where in the last bound we use~\eq{33} which implies $B'(x)>0$. So
%
%
\begin{equation}
\label{34} B'(x)\kappa_b(x) \leq1.
\end{equation}
Now, from this we have for $x>x_0$
\[
\kappa_b'(x) = -1+B'(x)
\kappa_b(x) \leq0.
\]
Similarly, we have
%
%
\begin{equation}
\label{35} a\leq z \leq x \leq x_0\quad\Longrightarrow\quad B(z)
\geq B(x)\quad\mbox{and}\quad\bigl\llvert B'(z)\bigr\rrvert\geq\bigl\llvert
B'(x)\bigr\rrvert.
\end{equation}
So, using~\eq{35}, for $x<x_0$,
\[
\int_a^x e^{-B(z)} \,dz \leq\int
_a^x\frac{\llvert B'(z)\rrvert }{\llvert
B'(x)\rrvert }e^{-B(z)}\,dz =
\frac{e^{-B(x)}-e^{-B(a)}}{\llvert B'(x)\rrvert } \leq\frac
{e^{-B(x)}}{\llvert B'(x)\rrvert },
\]
thus
%
%
\begin{equation}
\label{36} \bigl\llvert B'(x)\bigr\rrvert
\kappa_a(x) \leq1,
\end{equation}
and so for $x<x_0$
\[
\kappa_a'(x) = 1+B'(x)
\kappa_a(x) \geq0.
\]
From~\eq{30}, we obtain
\[
\llVert f\rrVert\leq\llVert\tilde h\rrVert\sup_x\cases{
\kappa_a(x), &\quad if $x< x_0$,
\cr
\kappa_b(x),
&\quad if $x\geq x_0$.}
\]
Hence, having an increasing bound on $x<x_0$ and a decreasing
bound on $x>x_0$, implies that there is a maximum at $x_0$ and
\[
\llVert f\rrVert\leq\llVert\tilde h\rrVert\bigl(\kappa_a(x_0)
\vee\kappa_b(x_0) \bigr).
\]
The first bound of~\eq{32} now follows from the fact that $\kappa
_a(x_0)\vee
\kappa_b(x_0)\leq\kappa_a(x_0) + \kappa_b(x_0)$. The
second bound of~\eq{32} follows from
\eq{30} in combination with~\eq{34} and~\eq{36}. Using~\eq{31}, the
third bound of~\eq{32} follows in the same way.
\end{pf}

%
%
\begin{remark}
Lemma~\ref{lem9} applies to the standard normal distribution in which case
$B(x)=x^2/2$, $x_0=0$, and $C_B=(2\pi)^{-1/2}$ and~\eq{32} implies
\[
\llVert f_h\rrVert\leq\llVert\tilde h\rrVert\sqrt{2\pi}, \qquad
\bigl\llVert f'_h\bigr\rrVert\leq2\llVert\tilde h
\rrVert.
\]
The best-known bounds are given in \citet{Chen2011}, Lemma~2.4,
which improve the first
bound by a factor of $2$ and match the second. In the special case of
the form $h(\cdot)=\I[\cdot\leq t]$,
\citet{Chen2011}, Lemma~2.3, matches the bound of
Lemma~\ref
{lem9} of $\llvert xf_h(x)\rrvert \leq\llVert\tilde
h\rrVert $.
\end{remark}

Though not used below explicitly, we record the following theorem
summarizing the
utility of the lemmas above.
%
%
\begin{theorem}
Let $B$ be convex on $(a,b)$ with unique minimum $x_0$, $Z\sim P_B$,
and $W$ be a random variable on $(a,b)$.
If $\mathcal{F}$ is the set of functions on $(a,b)$ such that
for $f\in\mathcal{F}$
\[
\llVert f\rrVert\leq\frac{e^{B(x_0)}}{C_B}, \qquad\bigl\llVert B'f
\bigr\rrVert\leq1, \qquad\bigl\llVert f'\bigr\rrVert\leq2,
\]
then
\[
\sup_{t\in(a,b)} \bigl\llvert\IP[Z\leq t]-\IP[W\leq t]\bigr\rrvert
\leq\sup_{f\in
\mathcal{F}} \bigl\llvert\IE\bigl\{f'(W)-B'(W)f(W)
\bigr\} \bigr\rrvert.
\]
\end{theorem}
\begin{pf}
For $t\in(a,b)$, if $h_t(x)=\I[x\leq t]$, then taking the expectation
in~\eq{28} implies that
%
%
\begin{equation}
\IP[W\leq t]-\IP[Z\leq t]= \IE\bigl\{f_t'(W)-B'(W)f_t(W)
\bigr\}, \label{37}
\end{equation}
where $f_t$ satisfies~\eq{28} with $h=h_t$.
Taking the absolute value
and the supremum over $t\in(a,b)$ on both sides of~\eq{37}, we find
\[
\sup_{t\in(a,b)} \bigl\llvert\IP[W\leq t]-\IP[Z\leq t]\bigr\rrvert
= \sup_{t\in
(a,b)} \bigl\llvert\IE\bigl\{f'_t(W)-B'(W)f_t(W)
\bigr\} \bigr\rrvert.
\]
The result follows since $h_t(x)\in[0,1]$ implies $\llVert\tilde
h\rrVert \leq1$,
and so by Lemma~\ref{lem9}, $f_t\in\mathcal{F}$ for all $t\in(a,b)$.
\end{pf}

Finally, we will need the following two lemmas to develop Stein's
method. The
proofs are standard, and can be easily adopted from the normal case;
see, for example, \citet{Chen2005} and Rai{\v{c}} (\citeyear{Raic2003}).

%
%
\begin{lemma}[(Smoothing inequality)]\label{lem10} Let $B$ be convex
on $(a,b)$ with unique minimum $x_0$ and let $Z\sim P_B$. Then, for
any random variable $W$ taking values in $(a,b)$ and for any $\eps>0$,
we have
\[
\dk\bigl(\law(W),\law(Z) \bigr) \leq\sup_{a<s<b}\bigl\llvert
\IE h_{s,\eps
}(W)-\IE h_{s,\eps}(Z)\bigr\rrvert+
C_Be^{-B(x_0)}\eps,
\]
where
%
%
\begin{equation}
\label{haha} h_{s,\eps}(x) = \frac{1}{\eps} \int
_0^\eps\I[ x\leq s+u]\,du.
\end{equation}
\end{lemma}

%
%
\begin{lemma}[(Bootstrap concentration inequality)]\label{lem11} Let $B$
be convex on $(a,b)$ with unique
minimum $x_0$ and let $Z\sim P_B$. Then, for any random variable $W$ taking
values in $(a,b)$, for any $a<x<b$, and for any $\eps>0$, we have
\[
\IP[s\leq W\leq s+\eps] \leq C_Be^{-B(x_0)}\eps+ 2\dk\bigl(
\law(W),\law(Z) \bigr).
\]
\end{lemma}

\subsection{Application to the generalized gamma distribution}\label{sec5}

We use the general results of Section~\ref{sec4} to prove the following
more explicit statement of Theorem~\ref{thm4} for the generalized
gamma distribution.
%
%
\begin{theorem}\label{thm5} Let\vspace*{1pt} $Z\sim\GG(\alpha,\beta)$ for
some $\alpha\geq1, \beta\geq1$ and let $W$ be a nonnegative random
variable with $\IE W^\beta= \alpha/\beta$. Let $W^*$
have the
$(\alpha,\beta)$-generalized equilibrium transformation of
Definition~\ref{def1}.
If $\beta=1$ or $\beta\geq2$, then for all $0<b\leq1$,
\begin{eqnarray*}
&& \dk\bigl(\law(W), \law(Z) \bigr)
\\
&&\qquad  \leq b \bigl[10 M_{\alpha,\beta}+2
\beta(
\beta-1) \bigl(1+2^{\beta-2} \bigl(\IE W^{\beta-1}+b^{\beta-1}
\bigr) \bigr)M_{\alpha,\beta}'+4 \beta\IE W^{\beta-1} \bigr]
\\
&&\quad\qquad{}+4 \bigl(2+(\beta+\alpha-1)M_{\alpha,\beta}' \bigr)\IP\bigl[
\bigl\llvert W-W^*\bigr\rrvert>b\bigr],
\end{eqnarray*}
where here and below
\begin{eqnarray*}
M_{\alpha,\beta}&:=&\alpha^{1-1/\beta} \beta^{1/\beta}
e^{-4/9+{1}/(6\sklvfrac{\alpha-1}{\beta}+9/4)}
\biggl(2{ \frac{\alpha-1}{\beta}}+1 \biggr)^{-1/2}
\\
&\leq& {e^{1/e} \alpha
^{1-1/\beta}} { \biggl(2{ \frac{\alpha-1}{\beta}}+1 \biggr)^{-1/2}},
\\
M'_{\alpha,\beta}&:=& \sqrt{2\pi} e^{-{1}/(6\sklvfrac{\alpha
-1}{\beta}+9/4)}{ \biggl({
\frac{\alpha-1}{\beta}}+1/2 \biggr)^{1/2} \biggl({ \frac{\alpha-1}{\beta}}+1
\biggr)^{1/\beta}} {\alpha^{-1} }
\\
& \leq&\sqrt{2\pi} { \biggl({ \frac{\alpha-1}{\beta}}+1/2 \biggr)^{1/2}
\biggl({ \frac{\alpha-1}{\beta}}+1 \biggr)^{1/\beta}} {\alpha^{-1}}.
\end{eqnarray*}
If $1<\beta<2$, then for all $0<b\leq1$,
\begin{eqnarray*}
&& \dk\bigl(\law(W), \law(Z) \bigr)
\\
&&\qquad \leq b \bigl(10 M_{\alpha,\beta}+4 \beta\IE W^{\beta-1} \bigr)+2 \beta
b^{\beta-1}M_{\alpha,\beta}'
\\
&&\quad\qquad{} +4 \bigl(2+(\beta+
\alpha-1)M_{\alpha,\beta}' \bigr)\IP\bigl[\bigl\llvert W-W^*\bigr
\rrvert>b\bigr].
\end{eqnarray*}
\end{theorem}

%
%
\begin{remark}
For a given $\alpha$ and $\beta$, the constants in the theorem may
be sharpened.
For example, the case $\alpha=\beta=1$ of the theorem is the
exponential approximation result~(2.5) of
Theorem~2.1 of \citet{Pekoz2011}, but here with larger constants.
These larger
constants come from three sources: first, below we bound some maximums
of nonnegative numbers by sums for the sake of simple formulas
(only if all but one of the terms in the maximum is positive
is there any hope of optimality in the constants). Second, $M_{\alpha,
\beta}$ and $M_{\alpha, \beta}'$
arise from bounds on the generalized gamma density, achieved by using
both sides of the inequalities in Theorems~\ref{thm6}
and~\ref{thm7} below. These inequalities are not optimal at the same
value for each side, so some precision
could be gained by using the appropriate exact bounds on the density
which in principle are recoverable from the
work below, but not particularly informative. Finally,
in special cases more information about the Stein solution may be
obtained. For example, in \citet{Pekoz2011} the term $\llvert
g(W)-g(W^*)\rrvert $ that appears in
the proof of Theorem~\ref{thm5} is there bounded by~$1$, whereas
following Lemma~\ref{lem16},
our general bound specializes to $2\llVert g\rrVert \leq4.3$.
\end{remark}

In the notation of Section~\ref{sec4}, for the generalized gamma distribution
we have $\varphi_{\alpha,\beta}(x) = C e^{-B(x)}$, $x>0$ with $a=0$
and $b=\infty$, and
\[
B(x) = x^\beta-(\alpha-1)\log x,\qquad C=\frac{\beta}{\Gamma
(\alpha/\beta)}.
\]
If $\alpha\geq1$ and $\beta\geq1$, then $B$ has nonnegative second
derivative and is thus convex. Since
\[
B'(x) = \beta x^{\beta-1} -\frac{(\alpha-1)}{x},
\]
$B$ has a unique
minimum at $x_0= (\frac{\alpha-1}{\beta} )^{1/\beta}$. Hence,
\[
B(x_0) = \psi\biggl(\frac{\alpha-1}{\beta} \biggr)\qquad\mbox{with }
\psi(x) = x - x\log(x), \psi(0)=0,
\]
and
%
%
\begin{equation}
C e^{-B(x_0)} = C e^{-\psi((\alpha-1)/\beta)}. \label{38}
\end{equation}
In order to apply Lemmas~\ref{lem10} and~\ref{lem11},
we need to bound~\eq{38}, for which
we use the following
two results about the gamma function.
%
%
\begin{theorem}[{[\citet{Batir2008}, Corollary 1.2]}]\label{thm6}
For all $x\geq0$,
\[
\sqrt{2} e^{4/9} \leq\frac{\Gamma(x+1)}{x^x e^{-x-{1}/{(6x+9/4)}}\sqrt
{x+1/2}} \leq\sqrt{2\pi}.
\]
\end{theorem}
%
%
\begin{theorem}[{[\citet{Wendel1948}, (7)]}]\label{thm7}
If $x>0$ and $0\leq s \leq1$, then
\begin{eqnarray*}
&&\biggl(\frac{x}{x+s} \biggr)^{1-s}\leq\frac{\Gamma(x+s)}{x^s\Gamma
(x)}\leq1.
\end{eqnarray*}
\end{theorem}

%
%
\begin{lemma}\label{lem12}
If $C$, $B$, and $x_0$ are as above for the generalized gamma
distribution and $\alpha\geq1, \beta\geq1$, then $C e^{-B(x_0)}\leq
M_{\alpha,\beta}$
\end{lemma}
\begin{pf}
Using Theorem~\ref{thm6} with $x=(\alpha-1)/\beta$ in the
inequality below implies
%
%
\begin{eqnarray}\label{39}
\qquad e^{-B(x_0)} &=& \biggl({ \frac{\alpha-1}{\beta}} \biggr)^{(\alpha
-1)/\beta
}e^{-\vfrac
{\alpha-1}{\beta}}
\nonumber\\[-8pt]\\[-8pt]\nonumber
&\leq&\Gamma\biggl({ \frac{\alpha-1}{\beta}} +1 \biggr
){e^{-4/9+{1}/{(6\sklvfrac{\alpha-1}{\beta}+9/4)}}} \biggl(2{
\frac{\alpha-1}{\beta}}+1 \biggr)^{-1/2}.
\end{eqnarray}
Since $C=\beta/\Gamma(\alpha/\beta)$, Theorem~\ref{thm7}
with $x=\alpha/\beta$ and $s=1-1/\beta$ yields
\[
C \Gamma\biggl({ \frac{\alpha-1}{\beta}} +1 \biggr)\leq\alpha
^{1-1/\beta
}
\beta^{1/\beta},
\]
and combining this with~\eq{39}, the lemma follows.
\end{pf}

We can also now prove the following lemma which is
used in applying Lemma~\ref{lem9}.

%
%
\begin{lemma}\label{lem13}
If $B$ and $x_0$ are as above for the generalized gamma distribution
and $\beta\geq1, \alpha\geq1$, then
$e^{B(x_0)}{\Gamma(\alpha/\beta)}/{\beta}\leq M'_{\alpha,\beta}$.
\end{lemma}
\begin{pf}
Using Theorem~\ref{thm6} with $x=(\alpha-1)/\beta$ in the following
inequality, we find
%
%
\begin{eqnarray}\label{40}
\qquad e^{B(x_0)}&=& \biggl({ \frac{\alpha-1}{\beta}} \biggr)^{-(\alpha
-1)/\beta
}e^{\vfrac{\alpha-1}{\beta}}
\nonumber\\[-8pt]\\[-8pt]\nonumber
&\leq& \sqrt{2\pi} { e^{-{1}/{(6\sklvfrac{\alpha-1}{\beta
}+9/4)}} \biggl({ \frac{\alpha-1}{\beta}}+1/2
\biggr)^{1/2} } \Gamma\biggl({ \frac{\alpha-1}{\beta}}+1
\biggr)^{-1}.
\end{eqnarray}
Now, Theorem~\ref{thm7} with $x=\alpha/\beta$ and $s=1-1/\beta$ yields
\begin{eqnarray*}
&&\frac{\Gamma(\alpha/\beta)}{\Gamma(\sklvfrac{\alpha-1}{\beta
}+1 )}\leq\frac{r}{\alpha} \biggl(\frac{\alpha-1}{\beta
}+1
\biggr)^{1/\beta},
\end{eqnarray*}
and combining this with~\eq{40}, the lemma follows.
\end{pf}

Before proving Theorem~\ref{thm5}, we collect properties of the Stein solution
for the generalized gamma distribution, which, according to~\eq{27}
and~\eq{28} satisfies
%
%
\begin{eqnarray}
\label{41} &\displaystyle f(x) := f_h(x) =x^{1-\alpha}e^{x^\beta}
\int_0^x \tilde h(z)z^{\alpha-1}
e^{-z^\beta}\,dz,&
\nonumber\\[-8pt]\\[-8pt]\nonumber
&\displaystyle f'(x)+ \biggl(\frac{\alpha-1}{x}-\beta x^{\beta-1}
\biggr)f(x)=\tilde h(x).&
\end{eqnarray}
First, we record a straightforward application of Lemmas~\ref{lem9}
and~\ref{lem13}.
%
%
\begin{lemma}\label{lem14}
If $f$ is given by~\eq{41}, then
\[
\llVert f\rrVert\leq\llVert\tilde h\rrVert M'_{\alpha,\beta},
\qquad\bigl\llVert f'\bigr\rrVert\leq2\llVert\tilde h\rrVert.
\]
\end{lemma}

%
%
\begin{lemma}\label{lem15}
If $f$ is given by~\eq{41}, $x>0$, $\llvert t\rrvert \leq
b\leq1$,
and $x+t>0$, then for $\beta=1$ and $\beta\geq2$,
\begin{eqnarray*}
&& \bigl\llvert(x+t)^{\beta-1}f(x+t) - x^{\beta-1}f(x)\bigr\rrvert
\\
&&\qquad\leq\llVert\tilde{h}\rrVert b \bigl[ (\beta-1) \bigl(1+2^{\beta
-2}
\bigl(x^{\beta-1}+ b^{\beta-1}\bigr)\bigr) M'_{\alpha,\beta}
+ 2 x^{\beta-1} \bigr]=:\llVert\tilde{h}\rrVert C_{b, \alpha, \beta}(x).
\end{eqnarray*}
For $1<\beta<2$, we have
\begin{eqnarray*}
&& \bigl\llvert(x+t)^{\beta-1}f(x+t) - x^{\beta-1}f(x)\bigr\rrvert
\\
&&\qquad \leq
\llVert\tilde{h}\rrVert\bigl( b^{\beta-1} M'_{\alpha,\beta}
+ 2 b x^{\beta-1} \bigr)=:\llVert\tilde{h}\rrVert C_{b, \alpha, \beta}(x).
\end{eqnarray*}
\end{lemma}
\begin{pf}
Observe that for all $\beta\geq1$,
%
%
\begin{eqnarray}\label{42}
&& \bigl\llvert(x+t)^{\beta-1}f(x+t) - x^{\beta-1}f(x)\bigr\rrvert
\nonumber
\\
&& \qquad\leq\bigl\llvert(x+t)^{\beta-1} - x^{\beta-1}\bigr\rrvert
\bigl\llvert f(x+t)\bigr\rrvert+ x^{\beta-1}\bigl\llvert f(x+t) -
f(x)\bigr
\rrvert
\\
&& \qquad\leq\bigl\llvert(x+t)^{\beta-1} - x^{\beta-1}\bigr\rrvert
\llVert f\rrVert+ b x^{\beta-1} \bigl\llVert f'\bigr\rrVert.\nonumber
\end{eqnarray}
In all cases, we use Lemma~\ref{lem14} to bound the norms appearing
in~\eq{42}. For
the remaining term, if $\beta=1$, then $\llvert(x+t)^{\beta-1} -
x^{\beta-1}\rrvert =0$ and the result follows.

If $\beta\geq2$, then the mean value theorem implies
\begin{eqnarray*}
&&\bigl\llvert(x+t)^{\beta-1} - x^{\beta-1}\bigr\rrvert\leq\llvert t
\rrvert(\beta-1) \bigl(x+\llvert t\rrvert\bigr)^{\beta
-2} \leq b (\beta-1)
(x+b)^{\beta-2}.
\end{eqnarray*}
Since $\beta\geq2$,
\begin{eqnarray*}
&&(x+b)^{\beta-2}\leq\max\bigl\{1,(x+b)^{\beta-1} \bigr\} \leq\max
\bigl\{1, 2^{\beta-2}\bigl(x^{\beta-1}+b^{\beta-1}\bigr) \bigr\},
\end{eqnarray*}
where the last inequality is H\"older's, and the result in this case
follows by bounding the maximum
by the sum.

For $1< \beta< 2$, since $x^{\beta-1}$ is concave and
increasing, $\llvert(x+t)^{\beta-1} - x^{\beta-1}\rrvert $
is maximized
when $x=0$ and $t=b$ in
which case it equals $b^{\beta-1}$.
\end{pf}

%
%
\begin{lemma}\label{lem16}
If $f$ is given by~\eq{41}, and we define
%
%
\begin{eqnarray}
g(x) &=& f'(x) + \frac{\alpha-1}{x} f(x),\qquad x>0, \label{43}
\end{eqnarray}
then
%
%
\begin{eqnarray}
g(x)&=& \tilde h(x) + \beta x^{\beta-1}f(x), \label{44}
\end{eqnarray}
and for $\beta\geq1$,
\[
\llVert g\rrVert\leq\llVert\tilde{h}\rrVert\max\bigl\{2+(\alpha
-1)M'_{\alpha,\beta}, 1+\beta M'_{\alpha,\beta}
\bigr\} \leq\llVert\tilde h\rrVert\bigl(2+(\beta+\alpha-1)M'_{\alpha
,\beta}
\bigr).
\]
\end{lemma}
\begin{pf}
The fact that~\eq{43} equals~\eq{44} is a simple rearrangement of the
second equality of~\eq{41}.

For the bounds, if $x\geq1$, then~\eq{43} implies
\begin{eqnarray*}
&&\bigl\llvert g(x)\bigr\rrvert\leq\bigl\llVert f'\bigr\rrVert+(
\alpha-1)\llVert f\rrVert,
\end{eqnarray*}
and if $x\leq1$, then~\eq{44} implies
\begin{eqnarray*}
&&\bigl\llvert g(x)\bigr\rrvert\leq\llVert\tilde h\rrVert+\beta\llVert
f\rrVert
,
\end{eqnarray*}
so that
\begin{eqnarray*}
\llVert g\rrVert&\leq&\max\bigl\{ \bigl\llVert f'\bigr\rrVert+(
\alpha-1)\llVert f\rrVert,\llVert\tilde{h}\rrVert+\beta\llVert
f\rrVert\bigr
\},
\end{eqnarray*}
and the result follows from Lemma~\ref{lem14}.
\end{pf}

The purpose of introducing $g$ in Lemma~\ref{lem16} is
illustrated in the following lemma.
%
%
\begin{lemma}\label{lem17}
If $f$ is a bounded function on $[0,\infty)$ with bounded derivative
such that $f(0)=0$, $W\geq0$ is a random variable
with $\IE W^\beta= \alpha/\beta$, and $W^*$ has the $(\alpha
,\beta)$-generalized equilibrium distribution of $W$ as in
Definition~\ref{def1}, then
for $g(x)=f'(x)+(\alpha-1) x^{-1} f(x)$,
\[
\IE g\bigl(W^*\bigr)=\beta\IE W^{\beta-1} f(W).
\]
\end{lemma}
\begin{pf}
If $V_\alpha\sim\B(\alpha,1)$ is independent of $W^{(\beta)} $
having the $\beta$-power bias distribution of $W$, then we can set
$W^*=V_\alpha W^{(\beta)} $ and
%
%
\begin{eqnarray}\label{45}
\IE f'\bigl(W^*\bigr) &=&\IE f'\bigl(V_\alpha
W^{(\beta)}\bigr) =\frac{\beta}{\alpha}\IE W^\beta
f'(V_\alpha W)
\nonumber\\[-8pt]\\[-8pt]\nonumber
&=& \beta\IE W^{\beta
} \int_0^1 u^{\alpha-1}f'(uW) \,du.
\end{eqnarray}
The case $\alpha=1$ easily follows from performing the integration
in~\eq{45}, keeping in mind that $f(0)=0$.
If $\alpha>1$, similar to the computation of~\eq{45},
%
%
\begin{equation}
(\alpha-1) \IE\frac{f(W^*)}{W^*}=\beta(\alpha-1) \IE W^{\beta-1} \int
_0^1 u^{\alpha-2}f(uW) \,du. \label{46}
\end{equation}
Applying integration by parts in~\eq{46} and noting $f(0)=0$ yields
%
%
\begin{equation}
\qquad (\alpha-1) \IE\frac{f(W^*)}{W^*}=\beta\IE\biggl\{W^{\beta
-1} \biggl( f(W)
- W\int_0^1 u^{\alpha-1}
f'(uW) \,du \biggr) \biggr\},\label{47}
\end{equation}
and adding the right-hand sides of~\eq{45} and~\eq{47} yields the lemma.
\end{pf}

We are now in a position to prove our generalized gamma approximation result.

\begin{pf*}{Proof of Theorem~\ref{thm5}}
Let $\delta=
\dk(\law(W),\law(Z) )$ and let $h_{s,\eps}$ be the
smoothed indicators
defined at~\eq{haha} in Lemma~\ref{lem10}. From Lemmas~\ref{lem10}
and~\ref{lem12}, we have
for every $\eps>0$,
%
%
\begin{equation}
\delta\leq\sup_{s>0}\bigl\llvert\IE h_{s,\eps}(W)-\IE
h_{s,\eps}(Z)\bigr\rrvert+ M_{\alpha,\beta}\eps. \label{48}
\end{equation}
Fix $\eps$ and $s$, let $f$ solve the Stein equation given explicitly
by~\eq{41} with $h:=h_{s,\eps}$
and let $g$ be as in Lemma~\ref{lem16}. By Lemma~\ref{lem17},
\begin{eqnarray*}
\IE h(W) - \IE h(Z) &=&\IE\bigl\{f'(W)-B'(W)f(W)
\bigr\}
\\
&=&\IE\biggl\{f'(W)- \biggl(\beta W^{\beta-1}-
\frac{\alpha
-1}{W} \biggr)f(W) \biggr\}
\\
&=& \IE\bigl\{g(W) - \beta W^{\beta-1} f(W) \bigr\} = \IE\bigl\{g(W) -g
\bigl(W^*\bigr) \bigr\}.
\end{eqnarray*}
And we want to bound this last term since in absolute value it is equal
to the first part of the bound in~\eq{48}.
With $I_1 = \I[\llvert W-W^*\rrvert \leq b]$,
%
%
\begin{eqnarray}\label{49}
&& \bigl\llvert\IE\bigl\{g(W) -g\bigl(W^*\bigr) \bigr\} \bigr\rrvert
\nonumber\\[-8pt]\\[-8pt]\nonumber
&&\qquad \leq
2\llVert
g\rrVert\IP\bigl[\bigl\llvert W-W^*\bigr\rrvert>b\bigr] + \bigl
\llvert\IE
\bigl\{I_1 \bigl(g(W) -g\bigl(W^*\bigr) \bigr) \bigr\} \bigr\rrvert.
\end{eqnarray}
Note from the representation~\eq{44} of $g$, if $x>0$, $\llvert
t\rrvert \leq
b\leq1$, and $x+t>0$,
\begin{eqnarray*}
&&g(x+t) - g(x) = h(x+t)-h(x) + \beta\bigl((x+t)^{\beta-1}f(x+t) -
x^{\beta-1}f(x) \bigr)
\end{eqnarray*}
and since $\llvert h(x+t)-h(x)\rrvert \leq\eps^{-1} \int_{t \wedge
0}^{t\vee0} \I[s<x+u\leq s+\eps] \,du$,
we apply Lemma~\ref{lem15} to find
%
%
\begin{eqnarray}\label{50}
&& \bigl\llvert\IE\bigl\{I_1 \bigl(g(W) -g\bigl(W^*\bigr) \bigr)
\bigr\} \bigr\rrvert
\nonumber\\[-8pt]\\[-8pt]\nonumber
&&\qquad \leq\frac{1}{\eps}\sup_{s\geq0}\int
_{0}^b\IP[s < W + u \leq s + \eps] \,du +
C_{b,\alpha, \beta},
\end{eqnarray}
where $C_{b, \alpha, \beta}:=\IE C_{b, \alpha, \beta}(W)$
and $C_{b, \alpha, \beta}(x)$ is defined in Lemma~\ref{lem15}; and
observe that for $1<\beta<2$,
$ C_{b,\alpha, \beta}$ is bounded since $\IE W^{\beta-1} \leq(\IE
{W^\beta})^{(\beta-1)/\beta}$.

Now using Lemmas~\ref{lem11} and~\ref{lem12} to find
\[
\IP[s < W + u \leq s + \eps]\leq M_{\alpha,\beta}\eps+2\delta
\]
and combining~\eq{48},~\eq{49} and~\eq{50}, we have
\[
\delta\leq M_{\alpha,\beta}\eps+ 2\llVert g\rrVert\IP\bigl[\bigl
\llvert W-W^*
\bigr\rrvert> b\bigr] + C_{b,\alpha,
\beta}+b M_{\alpha,\beta}+2
\eps^{-1}b \delta.
\]
Applying Lemma~\ref{lem16} to bound $\llVert g\rrVert $,
setting $\eps= 4b$,
and solving for $\delta$ now yields the bounds of the theorem.
\end{pf*}

\section*{Acknowledgments}
A portion of the work for this project was completed when Nathan Ross  was at
University of California, Berkeley.
Erol A. Pek\"oz would like to thank the Department of Statistics and Applied
Probability, National University of Singapore, for its hospitality. We
also thank Shaun McKinlay for pointing us to the reference~\citet
{Pakes1992}, Larry Goldstein for the suggestion to study Rayleigh
limits in bridge random walks, Louigi Addario-Berry for helpful
discussions, Jim Pitman for detailed pointers to the literature
connecting trees, walks and urns, and the Associate Editor and referee
for their many detailed and useful comments which have
greatly improved this work.




%

\printaddresses
\end{document}